\pdfoutput=1
\documentclass[reqno]{shinyart}
\usepackage[utf8]{inputenc}
\usepackage{amsmath, latexsym,graphicx,makeidx,subeqnarray,epsfig,epstopdf,caption,subcaption,tikz, float,dsfont,shinybib,enumitem}
\usepackage[backend=biber]{biblatex}
\newtheorem{thr}{Theorem}[section]
\newtheorem{co}[thr]{Corollary}
\newtheorem{lm}[thr]{Lemma}
\newtheorem{pr}[thr]{Proposition}
\newtheorem{ex}[thr]{Example}
\newtheorem{defn}[thr]{Definition}
\newtheorem{rem}[thr]{Remark}
\begin{document}
\title{Generalized projections on general Banach spaces}
\author{Akhtar A. Khan\footnotemark[1]\and Jinlu Li\footnotemark[7]\and Simeon Reich\footnotemark[3]}  \maketitle
\renewcommand{\thefootnote}{\fnsymbol{footnote}}
\footnotetext[1] {School of Mathematical Sciences, Rochester Institute of Technology, Rochester, New York, 14623, USA. (aaksma@rit.edu)}
\footnotetext[7] {Department of Mathematics, Shawnee State University, Portsmouth, Ohio 45662, USA. (jli@shawnee.edu) }
\footnotetext[3] {Department of Mathematics, The Technion – Israel Institute of Technology, 32000 Haifa, Israel. (sreich@technion.ac.il) }
{\large
\begin{abstract}  In general Banach spaces, the metric projection map lacks the powerful properties it enjoys in Hilbert spaces. There are a few generalized projections that have been proposed in order to resolve many of the deficiencies of the metric projection. However, such notions are predominantly studied in Banach spaces with rich topological structures, such as uniformly convex Banach spaces. In this paper, we investigate two notions of generalized projection in general Banach spaces. Various examples are provided to demonstrate the proposed notions and the loss of structure in the generalized projections after migrating from specially structured Banach spaces to general Banach spaces. Connections between the generalized projection and the metric projection are thoroughly explored.  \\[3pt]
\noindent{\color{structure}Key words.}  Generalized Chebyshev set, generalized identical points, generalized metric projection, generalized projection,  generalized proximal set.\\
\noindent{\color{structure}2010 Mathematics Subject Classification.}\ 41A10, 41A50, 47A05, 58C06.
\end{abstract}
\section{Introduction}\label{KLR-S1}
Let $H$ be a  Hilbert space with norm $\|\cdot\|$  and let $K$ be a nonempty, closed, and convex subset of $H$. The metric projection $P_{\text{\tiny{K}}}:H\to K$ assigns to any $x\in H,$ the unique point $P_{\text{\tiny{K}}}x$ in $K$ such that
\begin{equation}\label{KLR-S1-E1}\|x-P_{\text{\tiny{K}}}x\|\leq \|x-z\|,\ \text{for every}\ z\in K.
\end{equation}

The projection map exemplifies one of the most important classes of nonlinear maps and has found numerous applications in optimization, approximation theory, inverse problems, variational inequalities, image processing, neural networks, machine learning, and others. Due to the richness of the Hilbertian structure, the projection map in a Hilbert space  enjoys immensely useful properties. For instance, the projection map is monotone (but not strongly monotone, in general), nonexpansive, and it renders an absolute best approximation to the underlying convex set. Moreover, a robust variational characterization holds, which has far-reaching consequences in many branches of applied mathematics. For example, the variational form transforms a variational inequality into a fixed point problem convenient for studying a wide variety of iterative algorithms.

On the other hand, many applied models, such as the identification problems where the regularization space is a Banach space, need the projection map in general Banach spaces. Consequently,  numerous authors studied the notion of the metric projection in Banach spaces.   Note that although the minimization problem \eqref{KLR-S1-E1} makes sense in general Banach spaces, the existence and uniqueness of solutions cannot be guaranteed. Unfortunately, the metric projection in Banach spaces does not possess many desirable properties. Moreover, the lack of a Hilbertian structure needs to be replaced by an interplay between the primal and the dual space by exploiting the properties of the duality map, which causes additional difficulties. Also, the desire to obtain a fixed point formulation of a variational inequality suggests defining the projection map from the dual space to the convex set. See \cite{Aba83,AkmNamVee15,Alb17,BalMarTei21,Bau03,Bou15,BroDeu72,Bui02,CheGol59,ChiLi05,Den01,DeuLam80,FitPhe82,Li04,Li04a,LiZhaMa08,Osh70,Pen05,PenRat98,Ric16,Sha16}, and the cited references.

Inspired by the shortcomings of the metric projection, the notions of generalized projection and generalized metric projection were proposed and used extensively in numerous diverse disciplines, see \cite{Alb93,Alb96}. Although there are notable exceptions (see \cite{Li04a,Li05}), the two notions are mainly studied in Banach spaces with favorable topological structures, such as uniformly convex and uniformly smooth Banach spaces. The basic properties of the generalized projection and the generalized metric projection and their connections are largely unknown in general Banach spaces.

The primary object of this research is to fill this void by studying the generalized projection and generalized metric projection in the framework of general Banach spaces. In Definition~\ref{KLR-S2.1-D1}, we introduce the generalized projection in an arbitrary Banach space onto a closed but not necessarily convex set and collect its basic properties in Proposition~\ref{KLR-S2.1-P1}. We demonstrate three cases in a general Banach space, where the generalized projection is set-valued, single-valued, and not defined. Then, focusing on the critical case where the underlying closed set is convex, we study the structure of the generalized projection map (see Proposition~\ref{KLR-S2.3-P1}). We also characterize the reflexivity of the Banach space using the generalized projection (see Theorem~\ref{KLR-S2.3-T1}). We prove a partial variational principle for the generalized projection  (see Theorem~\ref{KLR-S2.4-T1}) and utilize examples to show that the variational principle is only a sufficient but not a necessary condition. The generalized projection is a map from the dual space to the convex set. To have an analog of the metric projection, which is a map from the primal space to the convex set, we propose a new notion of generalized metric projection in a general Banach space onto a closed set (see Definition~\ref{KLR-S3.1-D1}). We propose and study generalized proximal sets and generalized Chebyshev sets in Banach space and relate them to the generalized metric projection (see Theorem~\ref{KLR-S3.2-T1}). We prove the monotonicity of generalized metric projection (see Proposition~\ref{KLR-S3.3-P1}) and establish a partial variational principle (see Theorem~\ref{KLR-S3.3-T2}). We study the connections between the generalized metric projection and the normalized duality map by proposing generalized identical points; we use this to characterize the strictly convex Banach spaces. The inverse of the normalized duality map is also studied in the context of the generalized metric projection.

The contents of the paper are organized into seven sections. After a brief introduction in Section~\ref{KLR-S1}, we study the generalized projection in Section~\ref{KLR-S2}. The focus of Sections~\ref{KLR-S3} is on the developments related to the generalized metric projection. Section~\ref{KLR-S4} is devoted to studying connections between the generalized projection and the metric projection. The focus of Section~\ref{KLR-S5} is on the generalized projection in the Banach space of all bounded continuous real-valued maps. In Section~\ref{KLR-S6} we study the variational principles for the metric projection. The paper concludes with some remarks and future research directions in Section~\ref{KLR-S7}.
\subsection{Notations}\label{KLR-S1.1} We will frequently use the Banach space $\ell_1$ of all absolutely summable real sequences with dual space $\ell_1^* =\ell_{\infty}$. Note that $\ell_1$ is neither reflexive nor strictly convex. By $c$, we denote the Banach space of convergent sequences $t=\{t_n\}$ of real numbers with the norm $\displaystyle \|t\|=\sup_{1\leq n< \infty}|t_n|$. The dual space $c^*=\ell_1.$ For any $x=(x_0,x_1,\ldots)\in \ell_1$ and $t=\{t_n\}\in c,$ we define the pairing by
$$\langle x,t\rangle=x_0\lim_{n\to \infty}t_n+\sum_{n=1}^{\infty}x_nt_n.$$
By $c_0$ we denote the closed subspace of $c$ containing all convergent real sequences with limit zero. See \cite{DunSch88} for details.

For $r>0$ and for $k\in \mathds{R}$, we define the following nonempty, closed, and convex subsets of $\ell_1:$
\begin{subequations}\label{KLR-S1.1-E1}
\begin{align}
S(r)&=S(\theta,r):=\{x\in \ell_1:\ \|x\|\leq r\}.\\
D(r) &:=\{z\in \ell_1:\ \|z\|= r\ \text{and}\ z\ \text{has nonnegative entries}\}.\\
T(k) &:=\left\{z=(z_1,z_2,\ldots)\in \ell_1:\ \sum_{n=1}^{\infty}z_n=k\right\}.
\end{align}
\end{subequations}
The set $S(r)$ is the closed ball in $\ell_1$ centered at $\theta$ and with radius $r$, $D(r)$ is the $r$-simplex in $\ell_1$, and $T(k)$ is a closed hyperplane in $\ell_1.$ For simplicity, for $r=1$, we use the notions $S,D$, and $T$.

We conclude this subsection by giving some properties of $\ell_1$, where we use the notion:
$$D^+(r):=\{z\in D(r):\ z\ \text{has all positive entries}\}.$$
\begin{lm}\label{KLR-S1.1-L1} For $r>0$ and for $\beta_r :=(r,r,\ldots)\in \ell_{\infty},$ we have
\begin{description}
\item[($a$)] $D(r)\subseteq S(r)\cap T(r)$.
\item[($b$)] For any $x=(x_1,x_2,\ldots)\in T,$ we have $\|x\|\geq 1$. Moreover, for $x=(x_1,x_2,\ldots)\in T,$ we have
\begin{equation}\label{KL-S2-SS2-E00}\|x\|=1\quad \Leftrightarrow\quad x\in D\ \text{and}\ \|x\|>1\quad \Leftrightarrow\quad x\in T\backslash D.
\end{equation}
\end{description}
\end{lm}
\section{The Generalized Projection $\pi$}\label{KLR-S2}
\subsection{The Definition and Basic Properties}\label{KLR-S2.1}
Let $B$ be a (real) Banach space, let $B^*$ be the dual of $B$, and let $\langle \cdot,\cdot\rangle_{\text{\tiny{\emph{B}}}}$ be the pairing between $B^*$ and $B$. In the following discussion, in a Banach space $N$, we will denote the origin by $\theta_{\text{\tiny{\emph{N}}}}$ and the norm by $\|\cdot\|_{\text{\tiny{\emph{N}}}}$. We will drop the subscript when there is no confusion about the space.

We recall that the normalized duality map $J:B\to 2^{B^*}\backslash \{\emptyset\}$ is defined by
$$J(x)=\{jx\in B^*|\ \langle jx,x\rangle_{\text{\tiny{\emph{B}}}}=\|jx\|_{\text{\tiny{\emph{B}}}^*}\|x\|_{\text{\tiny{\emph{B}}}}=\|x\|_{\text{\tiny{\emph{B}}}}^2=\|jx\|_{\text{\tiny{\emph{B}}}^*}^2 \}.$$
The normalized duality map $J$ has many valuable properties (see Section~\ref{KLR-S6}) and plays a vital role in many fields, including approximation theory, optimization, and variational inequalities.

We define a Lyapunov functional $V:B^*\times B\to \mathds{R}$ by the following formula:
\begin{equation}\label{KLR-S2.1-E1}
V(\phi,x):=\|\phi\|^2_{\text{\tiny{\emph{B}}}^*}-2\langle \phi,x\rangle_{\text{\tiny{\emph{B}}}}+\|x\|^2_{\text{\tiny{\emph{B}}}},\quad \text{for any}\ \phi\in B^*,\ \text{and}\ x\in B.
\end{equation}

We begin with the following definition.
\begin{defn} \label{KLR-S2.1-D1} Let $B$ be a Banach space with $B^*$ as its dual and let $C$ be a nonempty subset of $B$. We define a set-valued map $\pi_{\text{\tiny{\emph{C}}}}:B^*\to 2^C$ by
 \begin{equation}\label{KLR-S2.1-D1-E1}
\pi_{\text{\tiny{\emph{C}}}}\phi=\left\{u\in C:\ V(\phi ,u)=\inf_{y\in C}V(\phi,y)\right\},\quad \text{for any}\ \phi\in B^*.
\end{equation}
The map $\pi_{\text{\tiny{\emph{C}}}}$ is called the \textbf{generalized projection} from $B^*$ to $2^C.$ For any $\phi\in B^*$ such that $\pi_{\text{\tiny{\emph{C}}}}\phi \ne \emptyset$, each element of the set $\pi_{\text{\tiny{\emph{C}}}}\phi$ is called a generalized projection of $\phi$ onto $C$.
\end{defn}
\begin{rem}\label{KLR-S2.1-R1}
Assuming that $B$ is a uniformly convex and uniformly smooth Banach space, Alber~\cite{Alb96} employed the above Lyapunov functional $V$ to propose the generalized projection map from the dual space $B^*$  to a nonempty, closed, and convex subset of $B$. Our primary objective in this paper is to use the Lyapunov functional $V$ defined in \eqref{KLR-S2.1-E1} to propose and analyze an extension of the generalized projection map from uniformly convex and uniformly smooth Banach spaces to general Banach spaces. We also note that although the use of the Lyapunov functional $V$ to introduce the generalized projection is attributed to Alber~\cite{Alb93,Alb96}, Zarantonello~\cite{Zar84} introduced and analyzed the same functional to propose an extension of the metric projection on closed and convex sets in general reflexive Banach spaces and proved many exciting results.
\end{rem}

The following result collects some basic properties of the generalized projection map $\pi$.
\begin{pr}\label{KLR-S2.1-P1} Let $B$ be a Banach space with dual $B^*$ and let $C$ be a nonempty subset of $B$. Then:
\begin{description}
\item[($a$)] $V(\phi,x)$ is continuous.
\item[($b$)] $V(\phi,x)$ is convex in $\phi$ when $x$ is fixed, and convex in $x$ when $\phi$ is fixed.
\item[($c$)] For any $x\in B$ and any $\phi\in B^*$, we have
$$(\|\phi\|-\|x\|)^2\leq V(\phi,x)\leq (\|\phi\|+\|x\|)^2.$$
\item[($d$)] For any $x\in B$ and any $\phi\in B^*$, the following equivalence holds:
$$V(\phi,x)=0\quad \Leftrightarrow \quad \phi\in J(x).$$
\item[($e$)] The map $\pi_{\text{\tiny{\emph{C}}}}$ is fixed on $C$, that is, for any $y,z\in C$, we have
$$z\in \pi_{\text{\tiny{\emph{C}}}}(jy), \ \text{for some}\ jy\in Jy\quad \Leftrightarrow \quad jy\in Jz.$$
\item[($f$)] The map $\pi_{\text{\tiny{\emph{C}}}}$ is monotone on $B^*$, that is, for all $\phi_1,\phi_2\in B^*$, we have
$$\langle \phi_1-\phi_2,u_1-u_2\rangle\geq 0,\quad \text{for any}\ u_1\in \pi_{\text{\tiny{\emph{C}}}}(\phi_1)\ \text{and}\  u_2\in \pi_{\text{\tiny{\emph{C}}}}(\phi_2).$$
\end{description}
\end{pr}
\begin{proof} The proofs of ($a$)-($d$) are straightforward and hence omitted. To prove ($e$), we assume that $z\in \pi_{\text{\tiny{\emph{C}}}}(jy)$, for some $jy\in Jy.$ Since $y\in C,$ we have
$$V(jy,z)=\|jy\|^2-2\langle jy,z\rangle+\|z\|^2\leq V(jy,y)=\|jy\|^2-2\langle jy,y\rangle+\|y\|^2=0. $$
Then, from ($d$), we obtain $jy\in Jz.$ For the converse, assume that $jy\in Jz.$ Then, by using property ($d$) once again, we obtain $V(jy,z)=0,$ which confirms that $z\in \pi_{\text{\tiny{\emph{C}}}}(jy).$ The proof of ($f$) is similar to the proof of \cite[Property 6.b]{Alb96} and hence we omit the proof here.
\end{proof}
\subsection{Some Examples}\label{KLR-S2.2}
If $B$ is a reflexive, strictly convex, and smooth Banach space and $C$ is a nonempty, closed, and convex subset of $B$, then from Alber~\cite{Alb93,Alb96}, the generalized projection $\pi_{\text{\tiny{\emph{C}}}}:B^*\to C$ is a well-defined single-valued map. However, the generalized projection defined in \eqref{KLR-S2.1-D1-E1} is a set-valued map in general Banach spaces. Li~\cite[Example 1.4]{Li04} gives a situation where the generalized projection is an empty set. On the other hand, Li~\cite[Example 1.2 and 1.3]{Li04} presents situations where the generalized projection is a well-defined set-valued map. This subsection provides three examples to demonstrate the cases where the generalized projection is a set-valued map, a single-valued map, and an empty set.

In the following example, the generalized projection $\pi$ is a well-defined set-valued map:
\begin{ex}\label{KLR-S2.2-Ex1} Let $\gamma=(1,1,1,\ldots)\in \ell_1^*=\ell_{\infty}.$ Then: $\pi_{\text{\tiny{\emph{T}}}}(\gamma)=D$, where $D$ is given in \eqref{KLR-S1.1-E1}.
\end{ex}
\begin{proof} By the properties of the simplex $T$ (see \eqref{KLR-S1.1-L1}), we compute
\begin{align*}
V(\gamma,x)&=1-2+1=0,\quad \text{for any}\ x\in D,\\
V(\gamma,y)&=1-2+\|y\|>0,\quad \text{for any}\ y\in T\backslash D,
\end{align*}
which completes the proof of the claim.
\end{proof}

We now present a situation where the generalized projection $\pi$ is a single-valued map:
\begin{ex}\label{KLR-S2.2-Ex2} Let $\ell_1^+$ and $\ell_{\infty}^+$ be the subsets of elements with non-negative entries in $\ell_1$ and $\ell_{\infty}$, and $\ell_1^-$ and $\ell_{\infty}^-$ be the subsets of elements with non-positive entries  in $\ell_1$ and $\ell_{\infty}$, respectively. Then:
\begin{description}
\item[($a$)] $\pi_{\ell_1^+}(\gamma)=\theta,$ for any $\gamma\in \ell_{\infty}^{-}$.
\item[($b$)] $\pi_{\ell_1^{-}}(\gamma)=\theta,$ for any $\gamma\in \ell_{\infty}^{+}$.
\end{description}
\end{ex}
\begin{proof} ($a$) Let $\gamma\in \ell_{\infty}^{-}$ be an arbitrary fixed element. Then,
$$V(\gamma,x)=\|\gamma\|^2-\langle \gamma,x\rangle+\|x\|^2\geq \|\gamma\|^2+\|x\|^2,\quad \text{for any}\ x\in \ell_1^{+},$$
and $V(\gamma,\theta)=\|\gamma\|^2$, which proves ($a$). The proof of ($b$) is analogous and hence omitted.
\end{proof}

In the following example, the generalized projection $\pi$ is an empty set (undefined).
\begin{ex}\label{KLR-S2.2-Ex3} Let $\displaystyle x=\left(0,1,2^{-1},2^{-2},\ldots\right)\in \ell_1=c^*.$ Then, $\pi_{c_0}(x)=\emptyset.$
\end{ex}
\begin{proof} For any positive integer $m$, we define $w_m\in c_0$ such that its  first $m$ entries are $2$ and all other entries are $0$. That is, $w_m=(2,2,\ldots,2,0,0,\ldots)$. Then,
$$\langle x,w_m\rangle =0\times 2+\sum_{n=1}^m\frac{1}{2^{n-1}}2=4\left(1-2^{-m-1}\right),$$
implying that
\begin{equation}\label{KLR-S2.2-Ex3-E1}
V(x,w_m)=4-2\times 4\left(1-2^{-m-1}\right)+4\to 0,\ \text{as}\ m\to \infty.
\end{equation}
Next, we claim that $V(x,t)>0,$ for any $t\in c_0.$ If possible, assume that the claim is false and there exists $t=(t_1,t_2,\ldots)\in c_0$ such that $V(x,t)=0$. This, due to the fact that $(\|x\|-\|t\|)^2\leq V(x,t)$, implies that $2=\|x\|=\|t\|$. Thus, $-2\leq t_n\leq 2$ for $n=1,2,\ldots$. However, since $\displaystyle \lim_{n\to \infty}t_n=0$, there are infinitely many $n$ such that $t_n<2$. Consequently,
\begin{align*}
V(x,t)&=\|x\|^2-2\langle x,t\rangle+\|t\|^2=4-2\sum_{n=1}^{\infty}\frac{1}{2^{n-1}}t_n+4> 4-2\sum_{n=1}^{\infty}\frac{1}{2^{n-1}}2+4=0,
\end{align*}
which is a contraction to the assumption $V(x,t)=0$ and hence the claim that $V(x,t)>0,$ for any $t\in c_0$ is verified. This, in conjunction with \eqref{KLR-S2.2-Ex3-E1}, yields $\pi_{c_0}(u)=\emptyset.$ The proof is thus complete.
\end{proof}
\subsection{Generalized projection $\pi$ onto convex sets}\label{KLR-S2.3}
We shall now prove some properties of the generalized projection   $\pi$ on nonempty, closed, and convex sets. We begin with the following:
\begin{pr}\label{KLR-S2.3-P1} Let $B$ be a Banach space with dual $B^*$, and let $C\subset B$ be nonempty, closed, and convex. Then, for any $\phi\in B^*$, the set $\pi_{\text{\tiny{\emph{C}}}}(\phi)$ is closed, convex, and bounded, provided that $\pi_{\text{\tiny{\emph{C}}}}(\phi)\ne \emptyset$.
\end{pr}
\begin{proof} For a given fixed $\phi\in B^*$, assume that $\pi_{\text{\tiny{\emph{C}}}}(\phi)\ne \emptyset.$ Let $v\in C$ be arbitrary. Then, for any $u\in \pi_{\text{\tiny{\emph{C}}}}(\phi)$, we have
$$(\|\phi\|-\|u\|)^2\leq \|\phi\|^2-2\langle \phi,u\rangle+\|u\|^2\leq \|\phi\|^2-2\langle \phi,v\rangle+\|v\|^2 \leq (\|\phi\|+\|v\|)^2,$$
and since $\phi$ and $v$ are fixed, the boundedness of the set $\pi_{\text{\tiny{\emph{C}}}}(\phi)$ is evident.

To prove that $\pi_{\text{\tiny{\emph{C}}}}(\phi)$ is closed, we choose $\{x_n\}\subset \pi_{\text{\tiny{\emph{C}}}}(\phi)$ such that $x_n\to \bar{x}$ as $n\to \infty.$ Then,
\begin{align*}
\|\phi\|^2-2\langle \phi,\bar{x}\rangle+\|\bar{x}\|^2 &=\lim_{n\to \infty}\left(\|\phi\|^2-2\langle \phi,x_n\rangle+\|x_n\|^2\right)\\
&=\lim_{n\to \infty} \inf_{y\in C}\left(\|\phi\|^2-2\langle \phi,y\rangle+\|y\|^2\right)\\
&=\inf_{y\in C}\left(\|\phi\|^2-2\langle \phi,y\rangle+\|y\|^2\right),
\end{align*}
which confirms that  $\bar{x}\in \pi_{\text{\tiny{\emph{C}}}}(\phi)$, proving the closedness of $\pi_{\text{\tiny{\emph{C}}}}(\phi).$

Finally, to prove the convexity of $\pi_{\text{\tiny{\emph{C}}}}(\phi)$, we take $x_1,x_2\in \pi_{\text{\tiny{\emph{C}}}}(\phi)$ and $0\leq \lambda\leq 1$. Then,
\begin{align*}
\|\phi\|^2-2\langle \phi,\lambda x_1&+(1-\lambda)x_2\rangle+\|\lambda x_1+(1-\lambda)x_2\|^2 \\
&\leq \lambda \left(\|\phi\|^2-2\langle \phi,x_1\rangle+\|x_1\|^2 \right)+(1-\lambda) \left(\|\phi\|^2-2\langle \phi,x_2\rangle+\|x_2\|^2 \right)\\
&=\lambda \inf_{y\in C}\left(\|\phi\|^2-2\langle \phi,y\rangle+\|y\|^2 \right)+(1-\lambda)\inf_{y\in C} \left(\|\phi\|^2-2\langle \phi,y\rangle+\|y\|^2 \right)\\
&=\inf_{y\in C}\left(\|\phi\|^2-2\langle \phi,y\rangle+\|y\|^2 \right),
\end{align*}
proving that $\lambda x_1+(1-\lambda)x_2\in \pi_{\text{\tiny{\emph{C}}}}(\phi).$ Therefore, $\pi_{\text{\tiny{\emph{C}}}}(\phi)$ is convex. The proof is complete.
\end{proof}

We next characterize the reflexivity of the Banach space $B$ using the generalized projection $\pi$.
\begin{thr}\label{KLR-S2.3-T1} A Banach space $B$ with dual $B^*$ is reflexive, if and only if, for every nonempty, closed, and convex $C\subset B$, the generalized projection $\pi_{\text{\tiny{\emph{C}}}}(\phi)$ is well defined for each $\phi\in B^*,$ that is,
$$ \pi_{\text{\tiny{\emph{C}}}}(\phi)\ne \emptyset,\quad \text{for any}\ \phi\in B^*.$$
\end{thr}
\begin{proof} In a reflexive Banach space $B$, for any $\phi\in B^*,$ we have $\pi_{\text{\tiny{\emph{C}}}}(\phi)\ne \emptyset$, see \cite[Theorem~2.1]{Li04}. To prove the converse, we assume that the given Banach space $B$ is non-reflexive. Let $S$ be the closed unit ball and let $\partial S$ be the unit sphere in $B$. Then, $S$ is a nonempty, closed, and convex subset of $B$. By James' theorem, there exists $\phi\in B^*$ with $\|\phi\|=1$ and
\begin{equation}\label{KLR-S2.3-T1P-E1}
\langle \phi,y\rangle<1,\quad \text{for every}\ y\in \partial S.
\end{equation}
Since $\phi\in B^*$, for any $z\in S\backslash \{\theta\}$, we have  $\frac{z}{\|z\|}\in \partial S$ and hence by \eqref{KLR-S2.3-T1P-E1}, we get $\langle \phi,\frac{z}{\|z\|}\rangle<1. $ That is,
$$\langle \phi,z\rangle\leq \|z\|,\quad \text{for every}\ z\in S\backslash \{\theta\}.$$
By the above relationship and the fact that $\|\phi\|=1$, for every $z\in S\backslash \{\theta\}$, we obtain
\begin{equation}\label{KLR-S2.3-T1P-E2}
V(\phi,z)=\|\phi\|^2-2\langle \phi,z\rangle+\|z\|^2>1-2\|z\|+\|z\|^2,\quad \text{for every}\ z\in S\backslash \{\theta\}.
\end{equation}
However, we have $V(\phi,\theta)=\|\phi\|^2=1>0$. Therefore, using \eqref{KLR-S2.3-T1P-E2}, we deduce that
\begin{equation}\label{KLR-S2.3-T1P-E2-a} V(\phi,y)=\|\phi\|^2-2\langle \phi,y\rangle+\|y\|^2>0,\quad \text{for every}\ y\in S.
\end{equation}
Since $\phi\in B^*$ and since $S$ is symmetric with respect to $\theta$, we have
\begin{equation}\label{KLR-S2.3-T1P-E3}\inf_{y\in S}V(\phi,y)=\inf_{y\in S}\left(\|\phi\|^2-2\langle \phi,y\rangle+\|y\|^2 \right)=2-2\sup_{y\in S}\langle \phi,y\rangle=2-2\|\phi\|=0.
\end{equation}
We combine \eqref{KLR-S2.3-T1P-E2-a} and \eqref{KLR-S2.3-T1P-E3} to obtain $\pi_{\text{\tiny{\emph{C}}}}(\phi)=\emptyset.$ The proof is thus complete.
\end{proof}
\subsection{A Variational Characterization of the Generalized projection $\pi$}\label{KLR-S2.4}
If a Banach space $B$ is uniformly convex and uniformly smooth (or a reflexive, strictly convex and smooth), then the normalized duality mapping $J$ and the generalized projection $\pi$ are single-valued maps. In this case, the generalized projection  $\pi$ can be characterized by the so-called basic variational principle of $\pi$ (for uniformly convex and uniformly smooth Banach spaces, see Alber~\cite[Property 6. c]{Alb96}; for reflexive strictly convex and smooth Banach spaces, see Ibarakia and Takahashi~\cite{IbaTak07}). For completeness and comparison, we recall this celebrated variational principle:
\begin{thr}\label{KLR-S2.4-T1} Let $B$ be a uniformly convex and uniformly smooth Banach space with dual $B^*$ and let $C$ be a nonempty, closed, and convex subset of $B$. Then, for any $\phi\in B^*$, we have
\begin{equation}\label{KLR-S2.4-T1-E1}
z\in \pi_{\text{\tiny{\emph{C}}}}(\phi)\quad \Leftrightarrow\quad \langle \phi -Jz,z-y\rangle\geq 0,\quad\text{for all}\ y\in C.
\end{equation}
\end{thr}

For general Banach spaces, for the generalized projection $\pi$ defined in \eqref{KLR-S2.1-D1-E1}, the basic variational principle fails to hold. More precisely, in general, such a variational characterization is only a sufficient condition (see Theorem~\ref{KLR-S2.4-T2} below) and not a necessary condition (see Example~\ref{KLR-S2.4-Ex1}).
\begin{thr}\label{KLR-S2.4-T2}  Let $B$ be a Banach space with dual $B^*$ and let $C$ be a nonempty closed subset of $B$. Then, for any $\phi\in B^*$, if there is $jz\in Jz$ such that
\begin{equation}\label{KLR-S2.4-T2-E1}
\langle \phi-jz,z-y\rangle \geq 0,\quad \text{for all}\ y\in C,
\end{equation}
then $z\in \pi_{\text{\tiny{\emph{C}}}}(\phi)$
\end{thr}
\begin{proof} Assume that there is $jz\in Jz$ such that \eqref{KLR-S2.4-T2-E1} holds. Then, for all $y\in C,$ we have
\begin{align*}
V(\phi,y)-V(\phi,z)&=\left(\|\phi\|^2-2\langle \phi,y\rangle+\|y\|^2\right)-\left(\|\phi\|^2-2\langle \phi,z\rangle+\|z\|^2\right)\\
&=-2\langle \phi,y\rangle+\|y\|^2+2\langle \phi,z\rangle-\|z\|^2\\
&=-2\langle \phi,y\rangle+2\langle \phi,z\rangle-2\langle jz,z\rangle+\|jz\|^2+\|y\|^2\\
&\geq  -2\langle \phi,y\rangle+2\langle \phi,z\rangle-2\langle jz,z\rangle+2\|jz\|\|y\|\\
&\geq -2\langle \phi,y\rangle+2\langle \phi,z\rangle-2\langle jz,z\rangle+2\langle jz,y\rangle\\
&=2\langle \phi-jz,z-y \rangle\geq 0,
\end{align*}
and hence $V(\phi,z)\leq V(\phi,y)$, for all $y\in C.$ It follows that $z\in \pi_{\text{\tiny{\emph{C}}}}(\phi),$ which completes the proof.
\end{proof}

The next example shows that the converse of Theorem~\ref{KLR-S2.4-T2} is false.
\begin{ex}\label{KLR-S2.4-Ex1}  Let $\gamma=(3,1,0,0,\ldots)\in \ell_1^*=\ell_{\infty}$ and $z=(1,0,0,\ldots)\in D\subset T\subset \ell_1$, see \eqref{KLR-S1.1-E1}. Then:
\begin{description}
\item[($a$)] $z\in\pi_{\text{\tiny{\emph{T}}}}(\gamma)$.
\item[($b$)] There are some $\bar{y}\in T$ such that
$$\langle \gamma-jz,z-\bar{y}\rangle <0,\quad \text{for every}\ jz\in Jz.$$
\end{description}
\end{ex}
\begin{proof} We begin with a proof of ($a$). From $z=(1,0,0,\ldots)\in D\subset T$, we  compute
$$V(\gamma,z)=\|\gamma\|^2-2\langle \gamma,z\rangle +\|z\|^2=9-2\times 3+1=4.$$
In the following proof, we repeatedly use the property of $T$: $\|y\|\geq 1$ for any $y=(y_1,y_2,\ldots)\in T.$ For any $y=(y_1,y_2,\ldots)\in T$, we have
\begin{equation}\label{KLR-S2.4-Ex1-E1}
V(\gamma,y)=\|\gamma\|^2-2\langle \gamma,y\rangle+\|y\|^2=9-2(3y_1+y_2)+\|y\|^2.
\end{equation}
We shall now study three cases:\\
\textbf{Case 1.} $3y_1+y_2=3.$ From \eqref{KLR-S2.4-Ex1-E1}, we have
$$V(\gamma,y)=9-2(3y_1+y_2)+\|y\|^2=9-6+\|y\|^2\geq 4.$$
\textbf{Case 2.} $3y_1+y_2<3.$ Using \eqref{KLR-S2.3-T1P-E3}, we obtain
$$V(\gamma,y)=9-2(3y_1+y_2)+\|y\|^2>9-6+\|y\|^2\geq 4.$$
\textbf{Case 3.} $3y_1+y_2>3.$ In this case, we set $3y_1+y_2=a>3$. Since $\displaystyle \sum_{n=1}^{\infty}y_n=1$, we have
$$y_1+y_2=a-2y_1,\quad y_2=a-3y_1,\quad \text{and}\quad \sum_{n=3}^{\infty}y_n=1-a+2y_1.$$
By substituting the above equation into the expression for $V(\gamma,y)$, we get
\begin{align}
V(\gamma,y)&=9-2(3y_1+y_2)+\|y\|^2=9-2a+(|y_1|+|y_2|+\sum_{n=3}^{\infty}|y_n|)^2\notag\\
&\geq 9-2a+\left(|y_1|+|a-3y_1|+|1-a+2y_1|\right)^2.\label{KLR-S2.4-Ex1-E2}
\end{align}

We will now divide the case $3y_1+y_2=a>3$ into four subcases with respect to $y_1$. \\
\textbf{Subcase 3.1.} $y_1\leq 0.$ By \eqref{KLR-S2.4-Ex1-E2} and $a> 3,$ we have
\begin{align*}
V(\gamma,y)&\geq 9-2a+(|y_1|+|a-3y_1|+|1-a+2y_1|)^2= 9 - 2a + (-y_1 + a-3y_1 + a- 2y_1-1)^2\\
&= 9 - 2a + (2a-6y_1 -1)^2\geq 9 - 2a + (2a -1)^2\geq 10 - 6a + 4a^2>28
\end{align*}
\textbf{Subcase 3.2.} $0<y_1\leq \frac{a}{3}.$ In this case, by $a>3$, we have $\frac{a}{3}<\frac{a}{2}-\frac12$. It follows that $a-3y_1\geq 0$ and $a-2y_1-1>0.$ By \eqref{KLR-S2.4-Ex1-E2} and $a>3$, we have
\begin{align*}
V(\gamma,y)&\geq 9-2a+(|y_1|+|a-3y_1|+|1-a+2y_1|)^2= 9 - 2a + (y_1 + a-3y_1 + a- 2y_1-1)^2\\
&= 9 - 2a + (2a-4y_1 -1)^2\geq 9 - 2a + \left(2a -\frac{4a}{3}-1\right)^2\\
&=9 - 2a + \left(\frac{2a}{3}-1\right)^2 = 10 - \frac{10a}{3} + \frac{4a^2}{9}> 4.
\end{align*}
\textbf{Subcase 3.3.} $\displaystyle \frac{a}{3}<y_1\leq \frac{a}{2}-\frac12$. In this case, by $a>3$ and $\displaystyle \frac{a}{3}< \frac{a}{2}-\frac12,$ it follows that $a-3y_1<0$ and $a-2y_1-1\geq 0.$ By \eqref{KLR-S2.4-Ex1-E2} and $a>3$, we have
\begin{align*}
V(\gamma,y)&\geq 9-2a+(|y_1|+|a-3y_1|+|1-a+2y_1|)^2= 9 - 2a + (y_1 + 3y_1-a + a-2y_1-1)^2\\
&= 9 - 2a + (2y_1 -1)^2> 9 - 2a + \left(\frac{2a}{3}-1\right)^2= 10 - \frac{10a}{3} + \frac{4a^2}{3}>4.
\end{align*}
\textbf{Subcase 3.4.} $y_1>\frac{a}{2}-\frac12.$ We have $a-3y_1<0$ and $a-2y_1-1<0.$ By \eqref{KLR-S2.4-Ex1-E2} and $a>3$, we get
\begin{align*}
V(\gamma,y)&\geq 9-2a+(|y_1|+|a-3y_1|+|1-a+2y_1|)^2= 9 - 2a + (y_1 + 3y_1-a - a+ 2y_1+1)^2\\
&= 9 - 2a + (-2a+6y_1 +1)^2> 9 - 2a + \left(-2a+6\left(\frac{a}{2}-\frac12\right)+1\right)^2\\
&= 9 - 2a + (a -2)^2= 13 - 6a + a^2>4.
\end{align*}
Using $V(\gamma,z)=4$ and combining all the cases above, we obtain $z\in \pi_{\text{\tiny{\emph{T}}}}(\gamma)$.

Next, we show that $z=\pi_{\text{\tiny{\emph{T}}}}(\gamma)$, that is, it is a singleton. Suppose that $y=(y_1,y_2,\ldots)\in \pi_{\text{\tiny{\emph{T}}}}(\gamma)$. Then, from the proofs of the three cases above, we have that $V(\gamma,y)\geq 4.$ Since $V(\gamma,z)=4$, we must have $V(\gamma,y)=4$. Recall that in Cases~2 and 3 above, $V(\gamma,y)>4$. Hence $y$ must be in Case~1 and so it satisfies
$$3y_1+y_2=3,\quad \text{and}\quad \|y\|=\sum_{n=1}^{\infty}|y_n|=1.$$
Since $\displaystyle \sum_{n=1}^{\infty}y_n=1$, we have $0\leq y_n\leq 1$ for all $n$. Then, combining $3y_1+y_2=3$ and $y_1+y_2\leq 1$, we deduce that the only solution is $y_1=1$ and $y_n=0$ for all $n>1.$ That is, $y=z.$

Now we proceed to prove ($b$). By $z=(1,0,0,\ldots)\in \ell_1$ and $\|z\|=1$, we have
$$Jz=\{(1,\lambda_2,\lambda_3,\lambda_4,\ldots)\in \ell_{\infty}|\ -1\leq \lambda_n\leq 1,\ \quad \text{for}\ n=2,3,4,\ldots\}. $$
For any $jz=(1,\lambda_2,\lambda_3,\lambda_4,\ldots)\in Jz$ and for any $y=(y_1,y_2,\ldots)\in T$, we have
$$\langle \gamma-jz,z-y\rangle=2(1-y_1)+(1-\lambda_2)(-y_2)+(-\lambda_3)(-y_3)+(-\lambda_4)(-y_4)+\cdots$$
In particular, for any given $k>0$, by taking $\bar{y}=(1+k,0,-k,0,0,\ldots)\in T$, by $-1\leq \lambda_n\leq 1$, for $n=2,3,4,\ldots$, we have
\begin{align*}\langle \gamma-jz,z-\bar{y}\rangle&=2(1-(1+k))+(-\lambda_3)(-(-k))\leq -k<0,
\end{align*}
for any $jz=(1,\lambda_2,\lambda_3,\lambda_4,\ldots)\in Jz,$ which proves the claim.\end{proof}
\section{The Generalized Metric Projection $\Pi$}\label{KLR-S3}
\subsection{The Notion of the Generalized Metric Projection}\label{KLR-S3.1}
If $B$ is a Banach space but not a Hilbert space, the generalized projection $\pi_{\text{\tiny{\emph{C}}}}$ is a map from the dual space $B^*$ to a closed and convex set $C\subset B$. In this setting, $\pi_{\text{\tiny{\emph{C}}}}$ is entirely different from the metric projection $P_{\text{\tiny{\emph{C}}}}$, which is a map from $B$ onto $C$. Inspired by this discrepancy, in the following discussion we introduce a new generalized metric projection that maps a general Banach space $B$ onto $C$ and compare it to the metric projection $P_{\text{\tiny{\emph{C}}}}$. An analog of this notion was introduced in \cite{Alb96} for reflexive, strictly convex and smooth Banach spaces.
\begin{defn}\label{KLR-S3.1-D1} Let $B$ be a Banach space with dual $B^*$, let $C\subset B$ be nonempty, and let $J:B\to 2^{B^*}\backslash \{\theta\}$ be the  normalized duality map.  Define a set-valued map $\Pi_{\text{\tiny{\emph{C}}}}:B\to 2^C$ by
\begin{equation}\label{KLR-S3.1-D1-E1}
\Pi_{\text{\tiny{\emph{C}}}}(x)=\cup_{jx\in Jx}\pi_{\text{\tiny{\emph{C}}}}(jx),\quad \text{for any}\ x\in B.
\end{equation}
The set-valued mapping $\Pi_{\text{\tiny{\emph{C}}}}:B\to 2^C$ is called the generalized metric projection from $B$ to its subsets. If $\Pi_{\text{\tiny{\emph{C}}}}\ne \emptyset$, then every member of it is called a generalized metric projection of $x$ onto $C$.
\end{defn}

If $B$ is a Hilbert space and is identified with its dual space, $B^* = B$ with $J=I_{\text{\tiny{\emph{B}}}}$, (see the property ($J_2$) of the normalized duality mapping $J$ in the Appendix), then
$$V(Jx,y)=\|x-y\|^2,\quad \text{for any}\ x,y\in B.$$
Then, both $\pi_{\text{\tiny{\emph{C}}}}$ and $\Pi_{\text{\tiny{\emph{C}}}}$ coincide with the metric projection $P_{\text{\tiny{\emph{C}}}}$. Therefore, both the generalized projection $\pi_{\text{\tiny{\emph{C}}}}$ and the generalized metric projection $\Pi_{\text{\tiny{\emph{C}}}}$ are extensions of the metric projection  $P_{\text{\tiny{\emph{C}}}}$ from Hilbert spaces to Banach spaces. That is,
\begin{equation}\label{KLR-S3.1-E1}\pi_{\text{\tiny{\emph{C}}}}=\Pi_{\text{\tiny{\emph{C}}}}=P_{\text{\tiny{\emph{C}}}},\quad \text{if}\ B\ \text{is a Hilbert space and}\ C\subset B.
\end{equation}
Alber and Li~\cite{AlbLi07} explored the connections between the metric projection $P_{\text{\tiny{\emph{C}}}}$ and the generalized metric projection $\Pi_{\text{\tiny{\emph{C}}}}$, and by means of concrete examples in $\ell_p$, for $1<p<\infty$, showed that even in uniformly convex and uniformly smooth Banach spaces the two notions are different. Hence,
\begin{equation}\label{KLR-S3.1-E2}\pi_{\text{\tiny{\emph{C}}}} \ne \Pi_{\text{\tiny{\emph{C}}}},\quad \pi_{\text{\tiny{\emph{C}}}} \ne P_{\text{\tiny{\emph{C}}}},\quad \Pi_{\text{\tiny{\emph{C}}}} \ne P_{\text{\tiny{\emph{C}}}},\quad \text{if}\ B\ \text{is not a Hilbert space and}\ C\subset B.
\end{equation}
\begin{rem}On a smooth reflexive Banach space $B$, the generalized metric projection $\Pi$ coincides with the so-called Bregman projection, see \cite{Rei96}.
\end{rem}
\subsection{Generalized proximal sets and generalized Chebyshev sets}\label{KLR-S3.2}
In the optimization theory in Banach spaces, the metric projection  $P_{\text{\tiny{\emph{C}}}}:B\to 2^C$ has played a crucial role. One of the most critical issues in this field is finding the conditions that ensure that for the given Banach space $B$ and a set $C\subset B$, for every $x\in B,$ the set $P_{\text{\tiny{\emph{C}}}}(x)$ is nonempty and/or a singleton. The important notions of the proximal sets and the Chebyshev sets have been heavily explored to provide such conditions.

In the following discussion, to address similar questions for the generalized metric projection $\Pi_C$, we propose generalized proximal sets and generalized Chebyshev sets.
\begin{defn}\label{KLR-S3.2-D1} Let $B$ be a Banach space with dual $B^*$ and let $C$ be a nonempty subset of $B$. If
$$\Pi_{\text{\tiny{\emph{C}}}}(x)\ne \emptyset,\quad \text{for any}\ x\in B,$$
then $C$ is called a generalized proximal subset of $B$ for $\Pi_{\text{\tiny{\emph{C}}}}.$ Furthermore, if for any
$ x\in B$, $\Pi_{\text{\tiny{\emph{C}}}}(x)$ is a singleton, then $C$ is called a generalized Chebyshev subset of $B$ for $\Pi_{\text{\tiny{\emph{C}}}}.$
\end{defn}

It is of evident interest to study the structure of $\Pi_c(x)$, for $x\in B$, provided that $C$ is generalized proximal. However, before we embark on that issue, we will study the existence problem. We first construct a subset of a non-reflexive Banach space that fails to be generalized proximal.
\begin{ex}\label{KLR-S3.2-Ex1} For any positive integer $n$, we define $e_n\in \ell_1$ such that its $n$-th entry is $\frac{n+1}{n}$ and all other entries are $0$. Let $C=\overline{\text{co}}\{e_1,e_2,\ldots,e_n,\ldots\}.$ Then $C$ is a nonempty, closed, and convex subset of $\ell_1$ and $\Pi_{\text{\tiny{\emph{C}}}}(\theta)=\emptyset.$ Consequently, $C$ is not generalized proximal for $\Pi_{\text{\tiny{\emph{C}}}}.$
\end{ex}
\begin{proof} By the property ($J_3)$ of the duality map, we have $J\theta=\{\theta\}.$ By \eqref{KLR-S2.1-D1-E1} and \eqref{KLR-S3.1-D1-E1}, we obtain
\begin{align}
\Pi_{\text{\tiny{\emph{C}}}}(\theta)&=\pi_{\text{\tiny{\emph{C}}}}(\theta)=\{u\in C|\ \|\theta\|^2-2\langle \theta,u\rangle+\|u\|^2=\inf_{y\in C}\left(\|\theta\|^2-2\langle \theta,y\rangle+\|y\|^2\right)\}\notag\\
&=\{u\in C|\ \|u\|^2=\inf_{y\in C}\|y\|^2\}=\{u\in C|\ \|u\|=\inf_{y\in C}\|y\|\}.\label{KLR-S3.2-Ex1-E1}
\end{align}
For every $y\in C=\overline{\text{co}}\{e_1,e_2,\ldots,e_n,\ldots\}$, there is a sequence $\{\alpha_n\}\subset [0,1]$ with $\displaystyle \sum_{n=1}^{\infty}\alpha_n=1$ such that $\displaystyle y=\sum_{n=1}^{\infty} \alpha_ne_n$. Let $i$ be the smallest positive integer with $\alpha_i>0.$ Then, we have
\begin{align*}
\|y\|&=\sum_{n=1}^{\infty}\alpha_n\frac{n+1}{n}=\alpha_i\frac{i+1}{i}+\sum_{n\ne i}\alpha_n\frac{n+1}{n}\geq \alpha_i+\frac{\alpha_i}{i}+\sum_{n\neq i}\alpha_n=\frac{\alpha_i}{i}+\sum_{n=1}^{\infty}\alpha_n=\frac{\alpha_i}{i}+1>1,
\end{align*}
which leads to
\begin{equation}\label{KLR-S3.2-Ex1-E2}
V(\theta,y)=\|y\|>1,\quad \text{for all}\ y\in C.
\end{equation}
On the other hand, for $\{e_1,e_2,\ldots,e_n,\ldots\}\subset C$, we have
$$\lim_{n\to \infty}\|e_n\|=\lim_{n\to \infty}\frac{n+1}{n}=1,$$
which implies that
\begin{equation}\label{KLR-S3.2-Ex1-E3}
\inf_{y\in C}\left(\|\theta\|^2-2\langle \theta,y\rangle+\|y\|^2 \right)=\inf_{y\in C}\|y\|^2=1.
\end{equation}
By \eqref{KLR-S3.2-Ex1-E1}, combining \eqref{KLR-S3.2-Ex1-E2} and \eqref{KLR-S3.2-Ex1-E3}, we see that $\Pi_{\text{\tiny{\emph{C}}}}(\theta)=\emptyset.$
\end{proof}

Example~\ref{KLR-S3.2-Ex1} shows a nonempty closed and convex subset of $\ell_1$ which is not generalized proximal. This happens because of the lack of reflexivity of the Banach space $\ell_1$. The reflexivity of a Banach space is used later to prove the generalized proximal property.
\begin{thr}\label{KLR-S3.2-T1} Let $B$ be a Banach space with dual $B^*$ and let $C$ be a nonempty, closed, and convex subset of $B$. For any $x\in B$, if there exists $jx\in Jx$ such that $\pi_{\text{\tiny{\emph{C}}}}(jx)\ne \emptyset$, then $\Pi_{\text{\tiny{\emph{C}}}}(x)$ is a union of nonempty, closed, convex, and bounded subsets of $C$ and $\Pi_{\text{\tiny{\emph{C}}}}(x)$ is also bounded.
\end{thr}
\begin{proof} The first part follows immediately from Theorem~\ref{KLR-S2.3-T1} and Definition~\ref{KLR-S3.1-D1}. To prove that $\Pi_{\text{\tiny{\emph{C}}}}$ is bounded, we take an arbitrary fixed point $v\in C$. Then, for any $u\in \Pi_{\text{\tiny{\emph{C}}}}(x)$ with $u\in \pi_{\text{\tiny{\emph{C}}}}(jx)$, for some $jx\in Jx$, we have
$$(\|jx\|-\|u\|)^2\leq V(jx,u)\leq V(jx,v)\leq (\|jx\|+\|v\|)^2=(\|x\|+\|v\|)^2,\quad \text{for any}\ u\in \Pi_{\text{\tiny{\emph{C}}}}(x).$$
Since $x$ and $v$ are fixed, and $\|jx\|=\|x\|$, for any $jx\in Jx$, we deduce that $\Pi_{\text{\tiny{\emph{C}}}}(x)$ is bounded. \end{proof}

In the following result, we point out an implication of the reflexivity of a Banach space.
\begin{thr}\label{KLR-S3.2-T2} Let $B$ be a Banach space.
\begin{description}
\item[($a$)] If the space $B$ is reflexive, then every nonempty, closed, and convex subset $C$ of $B$ is generalized proximal for $\Pi_{\text{\tiny{\emph{C}}}}$
\item[($b$)] If the space $B$ is not reflexive, then there may be a nonempty, closed, and convex subset $C$ of $B$ that is not generalized proximal for $\Pi_{\text{\tiny{\emph{C}}}}$.
\end{description}
\end{thr}
\begin{proof} ($a$) Let the Banach space $B$ be reflexive and let $C$ be an arbitrary nonempty, closed, and convex subset of $B$. For any given $x\in B$, by property ($J_1$) of the normalized duality map, $J(x)$ is a nonempty, bounded, closed and convex subset of $B^*$. Then, for any $jx\in Jx$, from the first part of Theorem~\ref{KLR-S2.3-T1}, $\pi_{\text{\tiny{\emph{C}}}}(jx)\ne \emptyset$. By Definition~\ref{KLR-S3.1-D1}, it follows that $\Pi_{\text{\tiny{\emph{C}}}}(x)\ne \emptyset$, for any $x\in B$. This proves that $C$ is generalized proximal for $\Pi_{\text{\tiny{\emph{C}}}}$. The second part follows from Example~\ref{KLR-S3.2-Ex1}.
\end{proof}
\begin{co}\label{KLR-S3.2-C1}  Let $B$ be a reflexive and smooth Banach space and let $C\subset B$ be nonempty, closed, and convex. Then, for every  $x\in B$, $\Pi_{\text{\tiny{\emph{C}}}}(x)$ is a nonempty, closed, convex, and bounded subsets of $C$.
\end{co}
\begin{proof} From property ($J_{11}$) in the Appendix, the smoothness of $B$ implies that $J$ is a single-valued mapping. Thus, this corollary follows from Theorem~\ref{KLR-S2.3-T1} and Theorem~\ref{KLR-S3.2-T2}.
\end{proof}

If $B$ is a reflexive strictly convex and smooth Banach space, then $J$ and $\pi_{\text{\tiny{\emph{C}}}}$ are both single-valued. By Alber~\cite{Alb96}, and Ibarakia and Takahashi~\cite{IbaTak07}, we obtain the following result.
\begin{thr}\label{KLR-S3.2-T3} Let $B$ be a reflexive, strictly convex, and smooth Banach space. Then every nonempty, closed, and convex subset $C$ of $B$ is generalized Chebyshev for $\Pi_{\text{\tiny{\emph{C}}}}$.
\end{thr}
\subsection{Variational Properties of the generalized metric projection $\Pi_{\text{\tiny{\emph{C}}}}$}\label{KLR-S3.3}
This subsection studies the monotonic and variational properties of the generalized metric projection $\Pi_{\text{\tiny{\emph{C}}}}$ on general Banach spaces.
\begin{pr}\label{KLR-S3.3-P1} Let $B$ be a Banach space and let $C$ be a nonemtpy subset of $B$. Then, $\Pi_C$ is monotone in $B$. That is, for any $x,y\in B$, we have
\begin{equation}\label{KLR-S3.3-P1-EQ1}\langle jx-jy,u-v\rangle\geq 0,\quad \text{for any}\ u\in \pi_{\text{\tiny{\emph{C}}}}(jx)\subset \Pi_{\text{\tiny{\emph{C}}}}(x),\ \text{and}\ v\in \pi_{\text{\tiny{\emph{C}}}}(jy)\subset \Pi_{\text{\tiny{\emph{C}}}}(y).
\end{equation}
\end{pr}
\begin{proof} From $u\in \pi_{\text{\tiny{\emph{C}}}}(jx)$ and $v\in \pi_{\text{\tiny{\emph{C}}}}(jy)$, for $jx\in Jx$ and $jy\in Jy$, we have
\begin{align*}
\|jx\|^2-2\langle jx,u\rangle +\|u\|^2&\leq \|jx\|^2-2\langle jx,v\rangle +\|v\|^2,\\
\|jy\|^2-2\langle jy,v\rangle +\|v\|^2&\leq \|jy\|^2-2\langle jy,u\rangle +\|u\|^2,
\end{align*}
and by combining the above inequalities, we obtain
$$-\langle jy,u\rangle-\langle jx,v\rangle+\langle jy,v\rangle+\langle jx,u\rangle \geq 0, $$
which completes the proof.
\end{proof}
\begin{rem} Operators satisfying \eqref{KLR-S3.3-P1-EQ1} are called $d$-accretive in Alber and Reich~\cite{AlbRei94}.
\end{rem}
The generalized metric projection $\Pi$ is characterized by the following basic variational principle in uniformly convex and uniformly smooth Banach spaces (see Alber~\cite{Alb96}) and in reflexive, strictly convex and smooth Banach spaces (see \cite{IbaTak07}).
\begin{thr}\label{KLR-S3.3-T1} Let $B$ be a reflexive, strictly convex, and smooth Banach space, and let $C$ be a nonempty, closed, and convex subset of $B$. Then for any $x\in B$,
\begin{equation}\label{KLR-S3.3-T1-E1}
z\in \Pi_{\text{\tiny{\emph{C}}}}(x)\quad \Leftrightarrow \quad \langle Jx-Jz,z-y\rangle,\quad \text{for all}\ y\in C.
\end{equation}
\end{thr}

However, the generalized metric projection $\Pi$ for general Banach spaces is not characterized by the variational principle. More precisely, in general, \eqref{KLR-S3.3-T1-E1} is a sufficient condition for $z\in \Pi_{\text{\tiny{\emph{C}}}}(x)$ (see Theorem~\ref{KLR-S3.3-T2}), but it is not a necessary condition (see Example~\ref{KLR-S3.3-Ex1}).
 \begin{thr}\label{KLR-S3.3-T2} Let $B$ be a Banach space and let $C$ be a nonempty subset of $B$. For any given $x\in B$ and $jx\in Jx$, if there is $jz\in Jz$ such that
 \begin{equation}\label{KLR-S3.3-T2-E1}
 \langle jx-jz,z-y\rangle\geq 0,\quad \text{for all} \ y\in C,
 \end{equation}
 then $z\in \pi_{\text{\tiny{\emph{C}}}}(jx)\subset \Pi_{\text{\tiny{\emph{C}}}}(x).$
 \end{thr}
 \begin{proof} Assume that there are $jx\in Jx$ and $jz\in Jz$ satisfying \eqref{KLR-S3.3-T2-E1}. Then, for all $y\in C$, we have
 \begin{align*}
 V(jx,y)-V(jx,z)&=(\|jx\|^2-2\langle jx,y\rangle+\|y\|^2 )-(\|jx\|^2-2\langle jx,z\rangle+\|z\|^2 )\\
 &=-2\langle jx,y\rangle+\|y\|^2+2\langle jx,z\rangle-\|z\|^2\\
  &=-2\langle jx,y\rangle+2\langle jx,z\rangle-2\langle jz,z\rangle+\|jz\|^2+\|y\|^2\\
  &\geq -2\langle jx,y\rangle+2\langle jx,z\rangle-2\langle jz,z\rangle+2\|jz\|\|y\|\\
  &\geq -2\langle jx,y\rangle+2\langle jx,z\rangle-2\langle jz,z\rangle+2\langle jz,y\rangle\\
  &\geq 2\langle jx-jz,z-y\rangle\geq 0.
 \end{align*}
 That is, $V(jx,z)\leq V(jx,y)$ for all $y\in C.$ It follows that $z\in \pi_{\text{\tiny{\emph{C}}}}(jx)$ and hence $z\in \Pi_{\text{\tiny{\emph{C}}}}(x).$
 \end{proof}

 The following example, a modification of Example~\ref{KLR-S2.4-Ex1}, shows that \eqref{KLR-S3.3-T2-E1} is not a necessary condition for $z\in \pi_{\text{\tiny{\emph{C}}}}$, for any $jx\in Jx.$
 \begin{ex}\label{KLR-S3.3-Ex1} Let $u=(3,0,0,0,\ldots)\in \ell_1\backslash T.$ Let $\gamma=(3,1,0,0,\ldots)\in \ell_{\infty}$. Let $z=(1,0,0,0,\ldots)\in D\subset T.$ Then:
 \begin{description}
\item[($a$)] $\gamma\in Ju.$
\item[($b$)] $z=\pi_{\text{\tiny{\emph{T}}}} (\gamma)\subset \Pi_{\text{\tiny{\emph{T}}}} (u).$
\item[($c$)] There exists $\bar{y}\in T$ such that
$$\langle \gamma-jz,z-\bar{y}\rangle<0,\quad \text{for every}\ jz\in Jz. $$
\end{description}
 \end{ex}
 \begin{proof} The proof of part ($a$) is straightforward, and hence it is omitted here. The proofs of parts ($b$) and ($c$) are the same as the proofs of parts ($a$) and ($b$) in Example 2.9.
 \end{proof}

\section{Relating The generalized metric projection and the metric projection}\label{KLR-S4}
This section explores relationships between the generalized metric projection and the standard metric projection onto a nonempty subset of Banach spaces. We begin with the following fact.
\begin{pr}\label{KLR-S4-P1}  Let $B$ be a Banach space and let $C$ a nonempty subset of $B$. Then: $\Pi_{\text{\tiny{\emph{C}}}}(\theta)=P_{\text{\tiny{\emph{C}}}}(\theta).$
\end{pr}
\begin{proof} Since
$$V(\theta,y)=\|\theta\|^2-2\langle \theta,y\rangle+\|y\|^2=\|\theta -y\|^2,\quad \text{for any}\ y\in C, $$
the claim follows at once. \end{proof}
\subsection{Generalized identical points}\label{KLR-S4.1}
For any Banach space $B$ and an arbitrary given nonempty subset $C$ of $B$, the standard metric projection $P_{\text{\tiny{\emph{C}}}}:B\to 2^C$ always satisfies
\begin{equation}\label{KLR-S4.1-E1}
P_{\text{\tiny{\emph{C}}}} x=x,\quad \text{for every}\ x\in C.
\end{equation}
It is a significant property under the common sense of distance. However, in general, the generalized metric projection $\Pi_C:B\to 2^C$ does not satisfy \eqref{KLR-S4.1-E1}, which brings out the significant difference between $P_{\text{\tiny{\emph{C}}}}$ and $\Pi_{\text{\tiny{\emph{C}}}}$ on Banach spaces. Before further investigating the differences between the two notions, we introduce the notion of identical generalized points for the normalized duality mapping on Banach spaces.
\begin{defn}\label{KLR-S4.1-D1} Let $B$ be a Banach space with dual $B^*.$ For $x,y\in B$, if $Jx\cap Jy\ne \emptyset$, then $x$ and $y$ are called \textbf{generalized identical points}. For any $x\in B$, the set of all its generalized identical points is denoted by $\mathcal{J}(x)$. That is,
$$\mathcal{J}(x)=\{y\in B|\ Jx\cap Jy\ne \emptyset\}.$$
\end{defn}

It is clear that $x\in \mathcal{J}(x)$, for any $x\in B$.  Notice that, for any $x\in B$, we have $\|jx\|=\|x\|$, for any $jx\in Jx$. This implies that, $\|y\|=\|x\|$, for any $y\in \mathcal{J}(x)$. That is,
$$x \ \text{and}\ y\ \text{are generalized identical}\ \quad \Rightarrow\quad \|x\|=\|y\|,\quad \text{for}\ x,y\in B. $$

We have the following characterization of  strictly convex Banach spaces:
\begin{pr}\label{KLR-S4.1-P1} A Banach space B is strictly convex, if and only if,
$$\mathcal{J}(x)=x,\quad \text{for every}\ x\in B.$$
\end{pr}
\begin{proof} By the property $(J_7)$ of the normalized duality mapping (see the Appendix), $B$ is strictly convex, if and only if, $x\ne y$ implies that $J(x)\cap J(y)=\emptyset$, for every $x,y\in B.$
\end{proof}

\begin{lm} \label{KLR-S4.1-Ex4.3} For any given $r>0$ and for $b_r=(r,r,\ldots)\in \ell_{\infty}$, we have
$$x\in D^+(r)\quad \Rightarrow\quad \mathcal{J}x=\beta_r\quad \text{and}\quad x\in D(r)\backslash D^+(r)\quad \Rightarrow\quad \mathcal{J}x \varsupsetneq \{\beta_r\}.$$
\end{lm}

The following example demonstrates that if a Banach space is not strictly convex, there may be some points with non-singleton generalized identical sets.
\begin{ex}\label{KLR-S4.1-Ex1} Consider the infinite-dimensional simplex $D$ in $\ell_1,$ which is a nonempty, closed, and convex subset of $\ell_1.$ We recall that $D =\{x\in \ell_1|\ \|x\|=1,\ \text{and}\ x\ \text{has nonnegative entries}\}$. We let $\phi=(1,1,1,\ldots)\in \ell_{\infty}.$ Then,
\begin{description}
\item[($a$)] $\displaystyle \phi\in \cap_{x\in D}Jx.$
\item[($b$)] $\mathcal{J}(x)\supset D,$ for every $x\in D.$
\item[($c$)] $\Pi_{\text{\tiny{\emph{D}}}}(x)=D,$ for every $x\in D.$
\end{description}
\end{ex}
\begin{proof} For every $x\in D$, we have $\|\phi\|^2=\langle \phi,x\rangle=\|x\|^2=1,$ ensuring that $\phi\in Jx,$ for every $x\in D.$ Thus, for any $x,y\in D$, $x$ and $y$ are  generalized identical points. Let $x\in D$ be arbitrarily fixed. For every $y\in D$, from part ($a$), $V(\phi,y)=0$, which implies that $y\in \Pi_{\text{\tiny{\emph{D}}}}(x)$.
\end{proof}

Example~\ref{KLR-S4.1-Ex1} shows that, for a nonempty, closed, and convex subset $D$ of $\ell_1$, the generalized metric projection $\Pi_{\text{\tiny{\emph{D}}}}:D\to 2^{D}$ does not satisfy \eqref{KLR-S4.1-E1} for the standard metric projection $P_{\text{\tiny{\emph{D}}}}.$

The following theorem connects the generalized metric projection and the generalized identical points in Banach spaces. In view of Example~\ref{KLR-S4.1-Ex1}, it also shows the significant difference between $P_{\text{\tiny{\emph{C}}}}$ and $\Pi_{\text{\tiny{\emph{C}}}}$ on general Banach spaces.
\begin{thr}\label{KLR-S4.1-T1} Let $B$ be a Banach space and let $C$ be a nonempty subset of $B$. Then:
\begin{description}
\item[($a$)] $\Pi_{\text{\tiny{\emph{C}}}}(x)\supseteq \mathcal{J}(x)\cap C,$ for every $x\in B.$
\item[($b$)] $\Pi_{\text{\tiny{\emph{C}}}}(x)= \mathcal{J}(x)\cap C\ni x,$ for every $x\in C.$
\end{description}
\end{thr}
\begin{proof} ($a$) Since $\displaystyle \Pi_{\text{\tiny{\emph{C}}}}(x)=\cup_{jx\in Jx}\pi_{\text{\tiny{\emph{C}}}}(jx)$, for any fixed $x\in B$, we only need to prove that
$$\cup_{jx\in Jx}\pi_{\text{\tiny{\emph{C}}}}(jx)\supseteq \mathcal{J}(x)\cap C,\quad \text{for every}\ x\in B.$$
For any $z\in \mathcal{J}(x)\cap C$, there is $jx \in Jx\cap Jz$. Therefore,
$$V(jx,z)=\|jx\|^2-2\langle jx,z\rangle+\|z\|^2 =0\leq V(jx,y),\quad \text{for every}\ y\in C,$$
which confirms that $z\in \pi_{\text{\tiny{\emph{C}}}}(jx)$ and subsequently proving the desired containment:
$$\mathcal{J}(x)\cap C\subset \cup_{jx\in Jx}\pi_{\text{\tiny{\emph{C}}}}(jx).$$

As a particular implication of ($a$), we proved $\Pi_{\text{\tiny{\emph{C}}}}(x)\supseteq \mathcal{J}(x)\cap C,$ for every $x\in C.$ On the other hand, for every given $x\in C$ and for any $z\in \cup_{jx\in Jx}\pi_{\text{\tiny{\emph{C}}}}(jx)\subseteq C$, there exists $jx\in Jx$ such that $z\in \pi_{\text{\tiny{\emph{C}}}}(jx)$. Since $x\in C$, by \eqref{KLR-S2.1-D1-E1}, it follows that
$$V(jx,z)=\|jx\|^2-2\langle jx,z\rangle+\|z\|^2 \leq V(jx,x)=\|jx\|^2-2\langle jx,x\rangle+\|x\|^2=0,$$
and hence
$$0\leq (\|jx\|-\|z\|)^2\leq \|jx\|^2-2\langle jx,z\rangle+\|z\|^2=0,$$
which  implies that $\|jx\|=\|z\|$ and $\langle jx,z\rangle=\|jx\|^2=\|z\|^2. $ Consequently, $jx\in Jz$, ensuring that $jx\in Jx\cap Jz$ and hence $z\in \mathcal{J}(x)\cap C.$ Thus: $\mathcal{J}(x)\cap C\supset \cup_{jx\in Jx}\pi_{\text{\tiny{\emph{C}}}}(jx)=\Pi_{\text{\tiny{\emph{C}}}}(x).$ \end{proof}

We next give examples to show that, in a non-strictly convex Banach space $B$, there are $C\subset B$ and $x\in C$ such that $\Pi_{\text{\tiny{\emph{C}}}}(x)$ is not a singleton. That is, in general, $\{x\}\subset \Pi_{\text{\tiny{\emph{C}}}}(x),$ for $x\in C.$
\begin{ex}\label{KLR-S4.1-Ex2} Let $S$ be the closed unit ball of $\ell_1$ and let $D$ be the infinite dimensional simplex in $\ell_1$, which is a nonempty, closed, and convex subset of $S$. Then:
\begin{description}
\item[($a$)] $\Pi_{\text{\tiny{\emph{S}}}}(x)= \mathcal{J}(x)\cap S=D,$ for every $x\in D$ with all positive entries.
\item[($b$)] $\Pi_{\text{\tiny{\emph{S}}}}(x)=\mathcal{J}(x)\cap S\supset D,$ for every $x\in D$ with at least one zero entry.
\end{description}
\end{ex}
\begin{proof} ($a$). By Part ($b$) of Example~\ref{KLR-S4.1-Ex1} and Theorem~\ref{KLR-S4.1-T1}, we have
\begin{equation}\label{KL-P1-S4-SS1-Ex2-E1}
\Pi_{\text{\tiny{\emph{S}}}}(x)=\mathcal{J}(x)\cap S\supseteq D,\quad \text{for every}\ x\in D.
\end{equation}
On the other hand, let $x=(t_1,t_2,\ldots)\in D$ be an arbitrary fixed point such that all its coordinates are positive. For any $y\in S\backslash D$, and for any $jx\in Jx$, we have
$$0\leq V(jx,y)=\|jx\|^2-2\langle jx,y\rangle+\|y\|^2 =1-2\langle jx,y\rangle+\|y\|^2.$$
We will now study two cases for $\|y\|:$\\
\textbf{Case 1.} $\|y\|<1$. Then $V(jx,y)\geq 1-2\|y\|+\|y\|^2>0.$ Since $x\in D$ and $V(jx,x)=0,$ we get that
$$y\notin \Pi_{\text{\tiny{\emph{S}}}}(x),\quad \text{for}\ y\in S\backslash D,\ \text{with}\ \|y\|<1.$$
\textbf{Case 2.} $\|y\|=1.$ Let $y=(s_1,s_2,\ldots)$. Then $-1\leq s_n\leq 1$, for $n\in \mathds{N},$ and $0\leq t_n\leq 1$, for $n\in \mathds{N}.$ They satisfy $\displaystyle \|x\|=\sum_{n=0}^{\infty}t_n=1$ and $\displaystyle \|y\|=\sum_{n=1}^{\infty}|s_n|=1.$ For the given $jx\in Jx$, let $jx=(u_1,u_2,\ldots)\in \ell_1^*$ with $\|jx\|=1.$ Then $-1\leq u_n\leq 1$ for $n\in \mathds{N}$. From $\displaystyle 1=\langle jx,x\rangle=\sum_{n=0}^{\infty}t_nu_n$
 and $\displaystyle \sum_{n=1}^{\infty}t_n=1,$ $t_n>0$ for $n\in \mathds{N},$ we deduce that $u_n=1$, for $n\in \mathds{N}.$

 To obtain a contradiction, assume that $y\in \Pi_{\text{\tiny{\emph{S}}}}(x)$. From $V(jx,x)=0$, we have
 $$V(jx,y)=1-2\langle jx,y\rangle +1=0.$$
 From $y\in S\backslash D$, we have $\{n|\ s_n<0\}\ne \emptyset$. From the above equation and $u_n=1$, for $n\in \mathds{N}$, we have
 $$1=\langle jx,y\rangle=\sum_{n=0}^{\infty}s_nu_n=-\sum_{s_n<0} |s_n|+\sum_{s_n>0}s_n<\|y\|=1. $$
This contradiction shows that $y\notin \Pi_{\text{\tiny{\emph{C}}}}(x),$ for $y\in S\backslash D\ \text{with}\ \|y\|=1.$
Combining these two cases, we obtain $\Pi_{\text{\tiny{\emph{C}}}}(x)\cap S\backslash D=\emptyset.$  In view of \eqref{KL-P1-S4-SS1-Ex2-E1} and the above equation, part (a) is proved.

(b) Suppose $x=(t_1,t_2,\dots)\in D$ and $x$ has at least one zero entry. Define $j=(v_1,v_2,\ldots)\in \ell_1^*=\ell_{\infty}$ with $\|j\|=1$ by
$$v_n=\left\{\begin{array}{lll} 1 & \text{if} &t_n>0\\ -1 & \text{if} & t_n=0.\end{array}\right.$$
It is clear that $j\in Jx.$ Take $y=(s_1,s_2,\ldots)\in S$ with $\displaystyle \|y\|=\sum_{n=0}^{\infty}|s_n|=1$ such that
$$s_n=\left\{\begin{array}{lll} =0 & \text{if} & v_n=1\\ <0 & \text{if} & v_n=-1.\end{array}\right.$$
Then,
$$1=\|y\|=\sum_{n=0}^{\infty}|s_n|=-\sum_{v_n=-1}s_n+\sum_{v_n=1}0=\sum_{n=0}^{\infty}s_nv_n=\langle j,y\rangle=1=\|j\|,$$
confirming that $j\in Jx\cap Jy.$ So $y\in \mathcal{J}(x)\cap S.$ It is clear that $y\notin D.$ Thus, part (b) is proved.
 \end{proof}
\begin{ex}\label{KLR-S4.1-Ex3} Let $\phi=(1,1,1,\ldots)\in \ell_{\infty}$. Take three elements in $\ell_1$:
$$u=\left(\frac12,\frac{1}{2^2},\frac{1}{2^3},\ldots\right),\quad v=\left(\frac23,\frac{2}{3^2},\frac{2}{3^3},\ldots\right), \quad \text{and}\quad w=(8,0,0,\ldots).$$
Let $C=\overline{\text{co}}\{u,v,w\}$. Clearly, the set $C$ is a nonempty, closed, and convex subset of $\ell_1.$ Then:
\begin{description}
\item[($a$)] $\phi \in J(x)$, for every $x\in \overline{\text{co}}\{u,v\}.$
\item[($b$)] $\mathcal{J}(x)\cap C\supseteq \overline{\text{co}}\{u,v\}$, for every $x\in \overline{\text{co}}\{u,v\}.$
\item[($c$)] $\Pi_{\text{\tiny{\emph{C}}}}(x)\supseteq \overline{co}\{u,v\}$, for every $x\in \overline{\text{co}}\{u,v\}.$
\end{description}
\end{ex}
\begin{proof} The proof follows from Example~\ref{KLR-S4.1-Ex2} and Theorem~\ref{KLR-S4.1-T1}.
\end{proof}

The next result gives a sufficient condition to validate \eqref{KLR-S4.1-E1}.
\begin{co}\label{KLR-S4.1-C1} Let $B$ be a strictly convex Banach space and let $C$ be a nonempty subset of $B$. Then,
$$\Pi_{\text{\tiny{\emph{C}}}}(x)=x,\quad \text{for every}\ x\in C.$$
\end{co}
\begin{proof}By property $(J_7)$ of the normalized duality map, the strict convexity of $B$ implies that $J$ is one-to-one, that is, $J(x)\cap J(y)=\emptyset$, for any $x,y\in B$ with $x\ne y. $ The claim then follows at once from Theorem~\ref{KLR-S4.1-T1}.
\end{proof}
\begin{rem}\label{KLR-S4.1-R1}
By Corollary~\ref{KLR-S4.1-C1}, for any strictly convex Banach space $B$ and an arbitrary nonempty subset $C$ of $B$, the generalized metric projection $\Pi_{\text{\tiny{\emph{C}}}}$ coincides with the standard metric projection $P_{\text{\tiny{\emph{C}}}}$ on $C$. That is, $\Pi_{\text{\tiny{\emph{C}}}}(z)=P_C(z)$, for every $z\in C$ and if $B$ is strictly convex. Theorem~\ref{KLR-S4.1-T1} and Examples~\ref{KLR-S4.1-Ex1} and \ref{KLR-S4.1-Ex2} show that if $B$ is a non-strictly convex Banach space and $C$ is an arbitrary subset of $B$, the generalized metric projection $\Pi_{\text{\tiny{\emph{C}}}}$ may be different from the metric projection $P_{\text{\tiny{\emph{C}}}}$ on $C$. In general, $\Pi_{\text{\tiny{\emph{C}}}}(z)\ni P_C(z)=z$, for every $z\in C$ and $B$ is non-strictly convex. As studied by Alber and Li~\cite{AlbLi07}, even in uniformly convex and uniformly smooth Banach spaces,  $\Pi_{\text{\tiny{\emph{C}}}}$ is different from $P_{\text{\tiny{\emph{C}}}}$. From the above examples and theorems, for a general Banach space $B$ and a nonempty subset $C$ of $B$, we showed that, in general, $\Pi_{\text{\tiny{\emph{C}}}}\ne P_{\text{\tiny{\emph{C}}}}.$
\end{rem}

We will need the following result shortly:
\begin{lm}\label{KLR-S1.1-L2} For any $r>0$ and for $\beta_r :=(r,r,\ldots)\in \ell_{\infty},$ we have
$$x\in D^+(r)\quad \Rightarrow\quad Jx=\beta_r\quad \text{and}\quad x\in D(r)\backslash D^+(r)\quad \Rightarrow\quad Jx \varsupsetneq \{\beta_r\}.$$
\end{lm}

The next example, which follows from Theorem~\ref{KLR-S4.1-T1}, highlights, once again, the most significant differences between the generalized metric projection and the standard metric projection.
\begin{ex}\label{KLR-S4.1-Ex4} Let $r>0$ be arbitrary. Then:
\begin{description}
\item[($a$)] $\Pi_{\text{\tiny{\emph{S(r)}}}}(x)=\mathcal{J}(x)\cap S(r)=D(r)\ni x,$ for $x\in D(r)$.
\item[($b$)] $P_{\text{\tiny{\emph{S(r)}}}}(x)= x,$ for $x\in D(r)$.
\item[($c$)] $\Pi_{\text{\tiny{\emph{S(r)}}}}(y)= D(r),$ for any $y\in D^+(h)$ with $h>r$.
\item[($d$)] $P_{\text{\tiny{\emph{S(r)}}}}(y)\subset D(r),$ for any $y\in D^+(h)$ with $h>r$. That is, $\Pi_{\text{\tiny{\emph{S(r)}}}}(y)\ne \emptyset$ is closed, and convex subset of $D(r).$
\end{description}
\end{ex}
\begin{proof} Part ($a$) follows from Theorem~\ref{KLR-S4.1-T1}, whereas ($b$) is trivial. To prove ($c$), we pick $y\in D^+(h)$ and note that due to Lemma~\ref{KLR-S1.1-L2}, we have $Jx=\beta_h=(h,h,\ldots)\in \ell_{\infty}$, which is a singleton. It follows that
$$\Pi_{\text{\tiny{\emph{S(r)}}}}(y)=\pi_{\text{\tiny{\emph{S(r)}}}}(\beta_h),\quad \text{for any}\ y\in D^+(h).$$
For any $z=(z_1,z_2,\ldots)\in D(r)$, we have
\begin{equation}\label{KLR-S4.1-Ex4-E1}
V(\beta_h,z)=\|\beta_h\|^2-2\langle \beta_h,z\rangle+\|z\|^2=(h-r)^2.
\end{equation}
Clearly, $V(\beta_h,\theta)=h^2$. For any $w=(w_1,w_2,\ldots)\in S(r)\backslash \left(D(r)\cup \{\theta\}\right)$, we have
\begin{equation}\label{KLR-S4.1-Ex4-E2}
V(\beta_h,z)=\|\beta_h\|^2-2\langle \beta_h,w\rangle+\|w\|^2> \|\beta_h\|^2-2h\|w\|+\|w\|^2=(h-\|w\|)^2\geq (h-r)^2.
\end{equation}
By combining \eqref{KLR-S4.1-Ex4-E1} and \eqref{KLR-S4.1-Ex4-E2}, we establish ($c$).

Finally, we proceed to prove ($d$). For any $w=(w_1,w_2,\ldots)\in D(r)$, we have
\begin{equation}\label{KLR-S4.1-Ex4-E3}
\|y-w\|\geq \|y\|-\|w\|\geq h-r.
\end{equation}
We set $v=\frac{r}{h}y$. Since $y\in D^+(h)$ and $\frac{r}{h}<1$, we have $v\in D^+(r)$ and
\begin{equation}\label{KLR-S4.1-Ex4-E4}
\|y\|-\|v\|=\left\|y-\frac{r}{h}y\right\|=\left(1-\frac{r}{h}\right)h=h-r.
\end{equation}
Combining \eqref{KLR-S4.1-Ex4-E3} and \eqref{KLR-S4.1-Ex4-E4}, we get $\frac{r}{h}y\in P_{\text{\tiny{\emph{S(r)}}}}(x).$ Thus, $P_{\text{\tiny{\emph{S(r)}}}}(y)$ is a nonempty subset of $D(r)$. The space $\ell_1$ is not strictly convex, and it is straightforward to show that $P_{\text{\tiny{\emph{S(r)}}}}(x)$ is closed and convex.

Next, we show that $P_{\text{\tiny{\emph{S(r)}}}}(y)$ is a proper subset of $D(r)$. Since $y\in D^+(h)$, taking $y=(y_1,y_2,\ldots)$, it follows that $\displaystyle \lim_{n\to \infty}y_n=0.$ Then, there are positive integers $m<k$ such that $y_m>y_k>0.$ We define $e_i\in D(r)$ with $i$-th entry $r$ and all other entries $0$, for $i=m,k$. By \eqref{KLR-S4.1-Ex4-E3}, it follows that
$$h-r\leq \|y-e_m\|=\sum_{n\ne m,k}y_n+r-y_m+y_k<\sum_{n\ne m,k}y_n+r-y_k+y_m=\|y-e_k\|.$$
Then, $e_n\notin P_{\text{\tiny{\emph{S(r)}}}}(y)$ (even though we don't know if $e_m\in P_{\text{\tiny{\emph{S(r)}}}}(y)$). The proof is complete.
\end{proof}

For any non-strictly convex Banach space $B$ and nonempty subset of $B$, as a consequence of Theorem~\ref{KLR-S4.1-T1} and Example~\ref{KLR-S4.1-Ex4}, we have
$$P_{\text{\tiny{\emph{C}}}}(z)=z\in \Pi_{\text{\tiny{\emph{C}}}}(z),\quad \text{for every}\ z\in C.$$
\subsection{The inverse of the normalized duality map and the generalized metric projection $\Pi$ }\label{KLR-S4.2}
Let $B$ be a Banach space with dual $B^*$. For any $\phi \in B^*,$ we let $J^{-1}(\phi )=\{x\in B| \ \phi \in Jx\}.$ Then, for any $x\in J^{-1}\phi $, $\|x\|=\|\phi \|.$ $J^{-1}\phi $ is called the normalized duality inverse  of $\phi \in B^*.$

The following example illustrates this notion of the inverse of the normalized duality map.
\begin{ex}\label{KLR-S4.2-Ex1} For any $a>0$, take $\beta_a =(a,a,\ldots)\in \ell_{\infty}.$ Then,
$$J^{-1}\beta_a =\{y=(t_1,t_2,\ldots)\in \ell_1:\ \ t_n\geq 0,\quad \text{for all}\ n\ \text{and}\ \|y\|=a\}$$
For $y=(y_1,y_2,\ldots)\in \ell_1$, if $y_n>0$, for all $n$ and $\|y\|=a$, then $Jy$ is a singleton and $Jy=\beta_a .$
\end{ex}
\begin{proof} Let $y=(t_1,t_2,\ldots)\in \ell_1$. Suppose that $y\in J^{-1}\beta_a $. Then, $\beta_a \in Jy$, that is,
$$\|y\|^2=\langle \beta_a ,y\rangle=\|\beta_a \|^2=a^2.$$
Since $\displaystyle \langle \beta_a ,y\rangle=a\sum_{n=1}^{\infty} t_n=a^2$, we get $\displaystyle \sum_{n=1}^{\infty} t_n=a.$ By $\displaystyle \|y\|=\sum_{n=1}^{\infty} |t_n|=a=\sum_{n=1}^{\infty} t_n$, we get $t_n\geq 0$, for $n\in \mathds{N}.$

On the other hand, if $y=(t_1,t_2,\ldots)\in \ell_1$ satisfy $t_n\geq 0$, for all $n$ and $\|y\|=a$, then, we have $\|y\|^2=\langle \beta_a,y\rangle=\|\beta_a \|^2=a^2$. That is, $\beta_a \in Jy.$

Now let $y=(t_1,t_2,\ldots)\in \ell_1$ such that $t_n>0,$ for all $n$ and $\|y\|=a.$ Take an arbitrary $jy=(\lambda_1,\lambda_2,\ldots)\in Jy$. Then $\|jy\|^2=\langle jy,y\rangle=\|y\|^2=a^2$. It follows that $\|jy\|=\sup_{n\in N}|\lambda_n|=a.$ Then, using $\displaystyle \langle jy,y\rangle= \sum_{n=1}^{\infty} t_n\lambda_n=a^2,$ $\displaystyle \sum_{n=1}^{\infty}t_n=a$, we get $\lambda_n=a$, for all $n\in \mathds{N}.$ Thus, $jy=\beta_a .$ \end{proof}

We have the following simple observations.
\begin{pr}\label{KLR-S4.2-P0} Let $B$ be a Banach space and let $C$ be a nonempty subset of $B$. Then, for any $\phi\in B^*$, and for any $x,y\in J^{-1}\phi$, we have
$$\pi_{\text{\tiny{\emph{C}}}}(\phi)\ne \emptyset\quad \Rightarrow\quad \Pi_{\text{\tiny{\emph{C}}}}(x)\cap \Pi_{\text{\tiny{\emph{C}}}}(y)\ne \emptyset.$$
\end{pr}
\begin{pr}\label{KLR-S4.2-P1} Let $B$ be a Banach space with dual $B^*$ and let $\emptyset \ne C\subset B$. Then, for any $\phi \in B^*$,
$$\pi_{\text{\tiny{\emph{C}}}}(\phi )\subset \cap \{\Pi_{\text{\tiny{\emph{C}}}}(x):\ x\in J^{-1}\phi \}.$$
\end{pr}
\begin{proof}  For any $x\in J^{-1}\phi $, it follows that $\phi \in Jz$, implying that  $\pi_{\text{\tiny{\emph{C}}}}(\phi )\subset \Pi_C(x)$, for any $x\in J^{-1}\phi.$
\end{proof}

The following example is an application of Proposition~\ref{KLR-S4.2-P1}.
\begin{ex}\label{KLR-S4.2-Ex2} Let $L_1(\mathds{R})$ be the Banach space of absolutely integrable real functions on $\mathds{R}$ with dual space $L_1^*(\mathds{R})=L_{\infty}(\mathds{R})$. For any $x\in L_1(\mathds{R})$, we have
$$\|x\|=\int_{-\infty}^{\infty}|x(t)|\,dt.$$ Let $G$ be the closed unit ball in $L_1(\mathds{R})$. For $i=1,2$, let
$$H_i=\{x\in L_1(\mathds{R}): x(t)\geq 0,\ \text{for a.e.}\ t\ \text{and}\ \|x\|=i\}$$
Let $\phi $ be the constant function on $\mathds{R}$ with value $2$ satisfying $\phi \in L_1^*(\mathds{R})$. Then,
\begin{description}
\item[($a$)] $J^{-1}\phi =H_2.$
\item[($b$)] $H_1\subseteq G$ and $\pi_{\text{\tiny{\emph{G}}}}(\phi )=H_1.$
\item[($c$)] $H_1=\cap \{\Pi_{\text{\tiny{\emph{G}}}}(x)|x\in H_2\}.$
\end{description}
\end{ex}
\begin{proof} ($a$). For any $x\in H_2$, we have
$$\|x\|^2=\left(\int_{-\infty}^{\infty}|x(t)|\,dt\right)^2=\left(\int_{-\infty}^{\infty}x(t)\,dt\right)^2=\int_{-\infty}^{\infty}2x(t)\,dt=\langle \phi ,x\rangle=4=\|\phi \|^2,$$
implying that $\phi \in Jx$. Hence,
\begin{equation}\label{KLR-S4.2-Ex2-E1}
H_2\subset J^{-1}\phi .
\end{equation}
On the other hand, for any $y\in J^{-1}\phi $, it follows that $\phi \in Jy$. Then,
$$4=\|\phi \|^2=\langle \phi ,y\rangle=\int_{-\infty}^{\infty}2y(t)\,dt =\|y\|^2=\left(\int_{-\infty}^{\infty}|y(t)|\,dt\right)^2,$$
which implies that
$$\int_{-\infty}^{\infty}|y(t)|\,dt=\int_{-\infty}^{\infty}y(t)\,dt=2.$$
We have,
\begin{align*}
\int_{y(t)<0}y(t)\,dt+\int_{y(t)\geq 0}y(t)\,dt=\int_{-\infty}^{\infty}y(t)\,dt=\int_{-\infty}^{\infty}|y(t)|\,dt=-\int_{y(t)<0}y(t)\,dt+\int_{y(t)\geq 0}|y(t)|\,dt,
\end{align*}
and hence
$$\int_{y(t)<0}y(t)\,dt=-\int_{y(t)<0}y(t)\,dt.$$
That is, $y(t)\geq 0$, for a.e. $t$ and $\|y\|=2$. Consequently, $y\in H_2,$ and furthermore, $H_2\supset J^{-1}\phi$, which when combined with \eqref{KLR-S4.2-Ex2-E1} proves ($a$)

We now proceed to prove ($b$). For any $z\in H_1\subset G,$ we have
\begin{equation}\label{KLR-S4.2-Ex2-E3}V(\phi ,z)=\|\phi \|^2-2\langle \phi ,z\rangle+\|z\|^2=1,\quad \text{for any}\ z\in H_1.
\end{equation}
On the other hand, for any $y\in G\backslash H_1$, we have $\|y\|\leq 1,$ for which we will analyze two cases:\\
\textbf{Case 1:} $\|y\|=1.$ Then, from $y\in G\backslash H_1$, we have
$$\int_{-\infty}^{\infty}y(t)\,dt <\int_{-\infty}^{\infty}|y(t)|\,dt =1,$$
which implies that
\begin{align}
V(\phi ,y)&=\|\phi\|^2-2\langle \phi,y\rangle+\|y\|^2=4-2\int_{-\infty}^{\infty}2y(t)\,dt+1>4-4\int_{-\infty}^{\infty}|y(t)|\,dt+1=1.\label{KLR-S4.2-Ex2-E4}
\end{align}
\textbf{Case 2:} $\|y\|<1.$ Then,
 \begin{align}
V(\phi ,y)&=\|\phi\|^2-2\langle \phi,y\rangle+\|y\|^2=4-2\int_{-\infty}^{\infty}2y(t)\,dt+1\geq 4-4\|y\|+\|y\|^2=(2-\|y\|)^2>1.\label{KL-S4-SS2-Ex2-E5}
\end{align}
Combining \eqref{KLR-S4.2-Ex2-E3}, \eqref{KLR-S4.2-Ex2-E4}, and \eqref{KL-S4-SS2-Ex2-E5}, we get ($b$). Part ($c$) follows from ($a$),  ($b$), and Proposition 4.9. \end{proof}

Proposition~\ref{KLR-S4.2-P1} and Example~\ref{KLR-S4.2-Ex2} offer the following insight: For any non-zero integer $m$, define a function $u_m$ as follows:
$$u_m(t)=\left\{\begin{array}{rll} |m|,&\text{if} & m-\frac{1}{m}\leq t\leq m+\frac{1}{m},\\
0& \text{otherwise}. & \end{array}\right.$$
Then, $u_m\in H_2$, for $m=\pm1,\pm2,\ldots.$ By Example~\ref{KLR-S4.2-Ex2} or Proposition~\ref{KLR-S4.2-P1}, we have
$$\Pi_{\text{\tiny{\emph{G}}}}(u_m)=H_1,\quad \text{for}\ m=\pm1,\pm2,\ldots$$
As $|m|\to \infty$, the graphs of $u_m$ will be very far apart from each other. However, all of them have exactly the same generalized metric projection to the closed unit ball $G$ in $L_1(\mathds{R})$.
\subsection{The inverse of the normalized duality map on $\ell_1$ }\label{KLR-S4.3}
\begin{lm}\label{KLR-S4.3-L1} For any $r>0$, let $\beta_r=(r,r,\ldots)\in \ell_{\infty}=\ell_1^{*}$. Then: $J^{-1}\beta_r=D(r).$
\end{lm}
\begin{proof} Let $x=(t_1,t_2,\ldots)\in \ell_1$. Suppose that $x\in J^{-1}\beta_r$. It follows that $\beta_r\in Jx$. That is, $\|x\|^2=\langle \beta_r,x\rangle=\|\beta_r\|^2=r^2.$ Since
$\displaystyle \langle \beta_r,x\rangle=r\sum_{n=1}^{\infty}t_n=r^2$, we have $\displaystyle \sum_{n=1}^{\infty}t_n=r$. By $\displaystyle \|x\|=\sum_{n=1}^{\infty}|t_n|=r=\sum_{n=1}^{\infty}t_n$, we get $t_n\geq 0$, for $n\in \mathds{N}.$

On the other hand, if $x=(t_1,t_2,\ldots)\in \ell_1$ such that $t_n\geq 0,$ for all $n$, and $\displaystyle \sum_{n=1}^{\infty}t_n=r$, then we have $\|x\|^2=\langle \beta_r,x \rangle=\|\beta_r\|^2=r^2 $. That is, $\beta_r\in Jx.$ The proof is complete.
\end{proof}

By the definition of $J^{-1}$, we have the following property immediately, which are some connections between the generalized metric projection and the metric projection on the hyperplanes in $\ell_1$.
\begin{ex}\label{KLR-S4.3-Ex1}  Let $u=(3,0,0,0,\ldots)\in \ell_1\backslash T$, and $\gamma=(3,0,0,0,\ldots)\in \ell_{\infty}$. For an integer $m>1$, let
\begin{align*}
v_m&=(2,0,0,\ldots,-1,0,0,\ldots)\in T\subset \ell_1,\ \text{in which}\ -1\ \text{is the}\ m\text{th entry}.\\
\gamma_m&=(3,0,0,\ldots,-3,0,0,\ldots)\in \ell_{\infty},\ \text{in which}\ -3\ \text{is the}\ m\text{th entry}.
\end{align*}
Then:
\begin{description}
\item[($a$)] $\gamma_m\in Ju$, for every $m>1.$
\item[($b$)] $\left(J^{-1}\gamma_m \right)\cap \Gamma =v_m$, for every $m>1.$
\item[($c$)] $v_m=\pi_{\text{\tiny{\emph{T}}}}(\gamma)\subset \Pi_{\text{\tiny{\emph{T}}}}(u)$, for $m=2,3,\ldots$
\item[($d$)] $\Pi_{\text{\tiny{\emph{T}}}}(u)\supset \{v_m:\ m=2,3,\ldots\}.$
\item[($e$)] $P_{\text{\tiny{\emph{T}}}}(u)=\{ y=(y_1,y_2,\ldots)\in T|\ 1\leq y_1\leq 3,\ y_m\leq 0,\ m=2,3,\ldots \}\supset \{v_m|\ m=2,3,\ldots\}$
\item[($f$)] $\Pi_{\text{\tiny{\emph{T}}}}(u)\cap P_{\text{\tiny{\emph{T}}}}(u)\supset \{v_m|\ m=2,3,\ldots\}\cup \{(1,0,0,\ldots)\}.$
\end{description}
\end{ex}
\begin{proof} The proof of part ($a$) being straightforward, we proceed to prove ($b$). Note that
$$J^{-1}\gamma_2\supseteq \{u=(3,0,0,\ldots),(0,-3,0,0,\ldots),(2,-1,0,0,\ldots),(2.5,-0.5,0,0,\ldots),(1.8,-1.2,0,0,\ldots)\}.$$
We will only prove ($b$) for $m=2.$ All other cases are similarly proved. Let $x=(x_1,x_2,\ldots)\in J^{-1}\gamma_2$. Then, $\|\gamma_2\|^2=\langle \gamma_2,x\rangle=\|x\|^2=9.$ It follows that $$\|x\|=\sum_{n=1}^{\infty}|x_n|=3,\quad
|x_n|\leq 3,\quad n=1,2,\ldots,\quad \text{and}\quad  3x_1-3x_2=9.$$
We note that $|x_1|+|x_2|\leq \|x\|=3.$ If $|x_1|+|x_2|<3,$ then $3x_1-3x_2\leq 3(|x_1|+|x_2|)<9$, which is a contradiction. Therefore, we must have $|x_1|+|x_2|=3$, implying that $x_n=0$ for $n\geq 3.$ Summarizing, for $x\in J^{-1}\gamma_2$, for $x$ we must have
\begin{equation}\label{KLR-S4.3-Ex1-E1}
|x_1|+|x_2|=3,\quad  x_1-x_2=3,\quad |x_n|\leq 3,\ \text{for}\ n=1,2,\ldots\quad \text{and}\quad x_n=0,\ \text{for}\ n\geq 3.\end{equation}
However, for any $x=(x_1,x_2,\ldots)\in J^{-1}\gamma_2\cap T$, then, in addition to the above equations, $x$ must satisfy $x_1+x_2=1$. By solving the system of linear equations $x_1+x_2=1$, and $x_1-x_2=3$, we obtain $x_1=2$ and $x_2=-1$. That is, $x=v_2.$

We now proceed to prove part ($c$). We will only prove for $m=2.$ From part ($b$), we have $v_2=(2,-1,0,0,\ldots)\in J^{-1}\gamma_2$, which implies that $V(\gamma_2,v_2)=0$. It follows at once that $v_2\in \pi_{\text{\tiny{\emph{T}}}}(\gamma_2)$. Next, we show that $\pi_{\text{\tiny{\emph{T}}}}(\gamma_2)$ is a singleton. Suppose that $x=(x_1,x_2,\ldots)\in \pi_{\text{\tiny{\emph{T}}}}(\gamma_2)$. Since  $V(\gamma_2,v_2)=0$, we must have $V(\gamma_2,x)=0$, which implies that $x=(x_1,x_2,\ldots)\in \pi_{\text{\tiny{\emph{T}}}}\cap \left(J^{-1}\gamma_2\right).$ By \eqref{KLR-S4.3-Ex1-E1}, we have $x_1-x_2=3$ and $x_1+x_2=1$, $x_n=0$ for $n\geq 3.$ Solving for $x_1,x_2$ from the above equation, we obtain $x=v_2.$

Since ($d$) follows at once from ($a$) and ($c$), we shall now prove ($e$). For any $y=(y_1,y_2,\ldots)\in T$, for $u=(3,0,\ldots)\in \ell_1\backslash T$, we have
$$\|y-u\|=|3-y_1|+\sum_{n=2}^{\infty}|y_n|\geq |3-y_1|+\left|1-\sum_{n=2}^{\infty}y_n\right|=|3-y_1|+|1-y_1|\geq 2,$$
for any real $y_1$. One can check that for any $y=(y_1,y_2,\ldots)\in T$,
$$\|y-u\|=2\quad\Leftrightarrow\quad 1\leq y_1\leq 3,\ y_m\leq 0,\quad \text{for all}\ m=2,3,\ldots$$
It is easy to see that for every $m=2,3,\ldots,$ the elements $v_m$ satisfy the above conditions. Part ($f$) follows at once from ($d$) and ($e$). The proof is thus complete.
\end{proof}
\subsection{The generalized metric projection and the  metric projection in the Banach space $c$}\label{KLR-S4.4}
We will now study the connections between the generalized metric projection $\Pi$ and the standard metric projection $P$ on the Banach space $c$ of convergent sequences, see Section~\ref{KLR-S1.1}.

For an arbitrary given $r>0$, let $\beta_r=(r,r,\ldots)$. \textcolor[rgb]{0.00,0.00,0.00}{Then $\beta_r\in \ell_{\infty}\cap (c\backslash c_0)$.}
\begin{lm}\label{KLR-S4.4-L1} In the Banach space $c$, for any $r>0$, let $\beta_r=(r,r,\ldots)\in \ell_{\infty}\in c\backslash c_0$. Then
$$J\beta_r=D(r)\subseteq \ell_1.$$
\end{lm}
\begin{proof} For any $y=(y_0,y_1,y_2,\ldots)\in J\beta_r\subseteq \ell_1$, we must have
$$r^2=\|\beta_r\|^2=\langle y,\beta_r\rangle=\|y\|^2=r^2,$$
which implies that
$$r^2=\langle y,\beta_r\rangle=y_0\lim_{n\to \infty}r+\sum_{n=1}^{\infty}y_nr=ry_0+r\sum_{n=1}^{\infty}y_n=r^2,$$
or $\displaystyle \sum_{n=0}^{\infty}y_n=r$. From $y\in J\beta_r\subseteq \ell_1$, we have $\displaystyle r=\|\beta_r\|=\|y\|=\sum_{n=0}^{\infty}|y_n|$, which implies that
$0\leq y_n\leq r$, for $n=1,2,\ldots.$ Hence, $y=(y_0,y_1,y_2,\ldots)\in D(r)\subseteq \ell_1.$

For the converse, for any $y=(y_0,y_1,y_2,\ldots)\in D(r)\subseteq \ell_1,$ we evaluate,
$$\langle y,\beta_r\rangle=y_0\lim_{n\to \infty}r+\sum_{n=1}^{\infty}y_nr=ry_0+r\sum_{n=1}^{\infty}y_n=r\sum_{n=0}^{\infty}y_n=r^2=\|y\|^2=\|\beta_r\|^2.$$
Then, $y=(y_0,y_1,y_2,\ldots)\in J\beta_r\subseteq \ell_1$. Hence, for $\beta_r=(r,r,\ldots)\in c\backslash c_0$, we get $J\beta_r=D(r)\subseteq \ell_1.$
\end{proof}
\begin{lm}\label{KLR-S4.4-L2} Let $r>0$. If $g\in D(r)\subseteq \ell_1=c^*$ has infinitely many positive entries, then $\pi_{c_0}(g)=\emptyset.$
\end{lm}
\begin{proof} For any given $g\in D(r)\subseteq \ell_1=c^*$, let $g=(u_0,u_1,u_2,\ldots)$. Assume that $g$ has infinitely many positive entries. For any $s=(s_1,s_2,\ldots)\in c_0,$ if $s=\theta$, then $V(g,\theta)=r^2>0$. Next, we suppose $\|s\|>0$. Since $s\in c_0,$ there is $m>1$  such that $-\|s\|<s_n<\|s_n\|$, for all $n\geq m$. Then
\begin{align*}
V(g,s)&=\|g\|^2-2\langle g,s\rangle+\|s\|^2=r^2-2\left(u_0\lim_{n\to \infty}s_n+\sum_{n=1}^{\infty}u_ns_n\right)+\|s\|^2=r^2-2\sum_{n=1}^{\infty}u_ns_n+\|s\|^2\\
&=r^2-2\sum_{n=1}^m u_ns_n-2\sum_{n=m+1}^{\infty}u_ns_n+\|s\|^2\geq r^2-2\|s\|\sum_{n=1}^mu_n-2\sum_{m+1}^{\infty}u_n s_n+\|s\|^2\\
&> r^2-2\|s\|\sum_{n=1}^mu_n-2\|s_n\|\sum_{m+1}^{\infty}u_n+\|s\|^2=r^2-2\|s\|\sum_{n=1}^{\infty}u_n+\|s\|^2\\
&=r^2-2r\|s\|+\|s\|^2=(r-\|s\|)^2\geq 0,
\end{align*}
where the strict inequality is due to the hypothesis that $\displaystyle \sum_{n=m+1}^{\infty}u_n>0$ and $\|s\|>0.$ Hence,
\begin{equation}\label{-E1}
V(g,s)>0,\quad \text{for any}\ s\in c_0.
\end{equation}
For this fixed $g\in D(r)\subseteq \ell_1=c^*$ with infinitely many positive entries, for any positive integer $n$, define a point $e_n\in c_0$ as follows:
\begin{equation}\label{-E2}
\text{The }m\text{-th entry of}\ e_n=\left\{\begin{array}{lllll}r,&\text{if} & m\leq n&\text{and}&u_m>0,\\ 0&&\text{otherwise.}&&\end{array}\right.
\end{equation}
Then, as $n$ is large enough such that $\|e_n\|^2=r^2$ (in the $c$-norm), we have
\begin{equation}\label{-E3}
V(g,e_n)=\|g\|^2-2\langle g,e_n\rangle+\|e_n\|^2=r^2-2\left( u_00+r\sum_{m=1}^nu_n\right)+r^2=r^2-2r\sum_{m=1}^nu_n+r^2\to 0\ \text{as}\ n\to \infty.
\end{equation}
By combining \eqref{-E1} and \eqref{-E3}, we get $\pi_{c_0}(g)=\emptyset.$
\end{proof}
\begin{pr}\label{KLR-S4.4-P1} For any $r>0$, let $\beta_r=(r,r,\ldots)\in \ell_{\infty}\in c\backslash c_0$. We define two subsets of $c_0$ by:
\begin{align*}
S(\beta_r)&=\left\{ s\in c_0|\ \|\beta_r-s\|\leq r\right\}=\left\{ s\in c_0|\ \|\beta_r-s\|= r\right\}\\
Z(r)&=\{s=(s_1,s_2,\ldots)\in c_0|\ \{n\in \mathds{N}| s_n=r\}\ne \emptyset\}.
\end{align*}
Then:
\begin{description}
\item[($a$)] $P_{c_0}(\beta_r)=S(\beta_r).$
\item[($b$)] $\Pi_{c_0}(\beta_r)=Z(r).$
\end{description}
\end{pr}
\begin{proof} ($a$) We first consider $P_{c_0}(\beta_r)$. For any $t=(t_1,t_2,\ldots)\in c_0$, by $\displaystyle \lim_{n\to \infty}t_n=0$, it follows that:
\begin{equation}\label{KLR-S4.4-P1-E1}
\|\beta_r-t\|=\sup_{1\leq n<\infty}|r-t_n|\geq r,\quad \text{and}\quad \lim_{n\to \infty}|r-t_n|=r,\quad \text{for any}\ t\in c_0.
\end{equation}
By \eqref{KLR-S4.4-P1-E1}, we get that $\theta \in P_{c_0}(\beta_r).$ On the other hand, for any $s=(s_1,s_2,\ldots)\in S(\beta_r)$, we have $\|\beta_r-s\|\leq r.$ From \eqref{KLR-S4.4-P1-E1}, it follows that
\begin{equation}\label{KLR-S4.4-P1-E2}
\|\beta_r-s\|=r,\quad \text{for any}\ s\in S(\beta_r),
\end{equation}
which implies that $s\in P_{c_0}(\beta_r)$, for any $s\in S(\beta_r)$. Hence, $S(\beta_r)\subseteq P_{c_0}(\beta_r).$

For the converse, for any $u=(u_1,u_2,\ldots)\in P_{c_0}(\beta_r)$, by \eqref{KLR-S4.4-P1-E1} and \eqref{KLR-S4.4-P1-E2}, we must have $\|\beta_r-u\|=r.$ It follows that $u\in S(\beta_r)$. Thus, $S(\beta_r)\supseteq P_{c_0}(\beta_r),$ as required.

For ($b$), we consider $\Pi_{c_0}(\beta_r)$. Let
$$D_0(r)=\{z=(z_0,z_1,z_2,\ldots)\in D(r)|\ z\ \text{has finite number of positive entries}\}.$$
We see that $D_0(r)\subset D(r)\backslash D^+(r)$. For any $j=(y_0,y_1,y_2,\ldots)\in J\beta_r\subseteq \ell_1$, from Lemma~\ref{KLR-S4.4-L1}, we have $j\in D(r)\subseteq \ell_1.$ Then, for any $t=(t_1,t_2,\ldots)\in c_0$, by $\displaystyle \lim_{n\to \infty}t_n=0$, it follows that
$$\langle j,t\rangle=y_0\lim_{n\to \infty} t_n+\sum_{n=1}^{\infty}t_ny_n=\sum_{n=1}^{\infty}t_ny_n,$$
implying that
$$(r-\|t\|)^2=(\|j\|-\|t\|)^2\leq V(j,t)=\|j\|^2-2\langle j,t\rangle+\|t\|^2=\|j\|^2-2\sum_{n=1}^{\infty}t_ny_n+\|t\|^2. $$
It follows that $V(j,t)=0$ implies that $\|t\|=r$ and $\|t\|\ne r$ implies that $V(j,t)>0$, for $t=(t_1,t_2,\ldots)\in c_0.$ Now, for any $s=(s_1,s_2,\ldots)\in Z(r)$, there is at least one point $z=(0,z_1,z_2,\ldots)\in D_0(r)$ such that
$$z_n=\left\{\begin{array}{lll}>0 &\text{if} & s_n=r,\\ =0,&\text{if}& s_n\ne r,\end{array}\right.$$
and $\displaystyle \|z\|=\sum_{s_n=r}z_n=r$. Then
$$V(z,s)=\|z\|^2-2\langle z,s\rangle+\|s\|^2=r^2-2\sum_{s_n=r}rz_n+r^2 =0,$$
confirming that $s\in \pi_{c_0}(z)$. By Lemma~\ref{KLR-S4.4-L1}, $z=(0,z_1,z_2,\ldots)\in D_0(r)\subseteq D(r)=J\beta_r\subseteq \ell_1$. We obtain that $s\in \Pi_{c_0}(\beta_r).$ Hence,
\begin{equation}\label{KLR-S4.4-P1-E3}
Z(r)\subseteq \Pi_{c_0}(\beta_r).
\end{equation}
On the other hand, for any $s=(s_1,s_2,\ldots)\in \Pi_{c_0}(\beta_r)$, there is $g=(u_0,u_1,u_2,\ldots)\in D(r)=J\beta_r$ such that $s\in \pi_{c_0}(g)$, similar to (3.8) in the proof of Lemma~\ref{KLR-S4.4-L1}, we get  $V(g,s)\leq \inf\{V(g,t)|\ t\in Z(r)\}=0$. Since $Z(r)\subseteq c_0$, it follows that $V(g,s)=0.$ By $V(g,s)\geq (\|g\|-\|s\|)^2$, we get
\begin{equation}\label{KLR-S4.4-P1-E4}
\|g\|=\|s\|=r.
\end{equation}
By the assumption that $s\in \pi_{c_0}(g)$, from Lemma~\ref{KLR-S4.4-L1}, we deduce $g\in D_0(r)$ because $g\in D(r)\notin D_0(r)$, if and only if, $g$ has infinitely many positive entries. Then there is a positive integer $k$ such that $\displaystyle r=\|g\|=\sum_{n=0}^ku_n$ and $u_n=0$, for all $n>k$. Thus,
$$0=V(g,s)=r^2-2\left(u_0\lim_{n\to \infty}s_n+\sum_{n=1}^{\infty}s_nu_n\right)+r^2=r^2-2\sum_{n=1}^ks_nu_n+r^2,$$
implying that
\begin{equation}\label{KLR-S4.4-P1-E5}
\sum_{n=1}^ks_nu_n=r^2.
\end{equation}
By \eqref{KLR-S4.4-P1-E4} and $g\in D_0(r)$,we have $0\leq u_n\leq r$, $-r\leq s_n\leq r$ and $\displaystyle \sum_{n=0}^ku_n=r$. From \eqref{KLR-S4.4-P1-E5}, we obtain $u_n>0$ giving $s_n=r$. By $\displaystyle \sum_{n=0}^ku_n=r$, we get $\{n\in \mathds{N}| \ s_n=r\ne \emptyset \}$. That is, $s\in Z(r)$. Hence
\begin{equation}\label{KLR-S4.4-P1-E6}
\Pi_{c_0}(\beta_r)\subseteq Z(r).
\end{equation}
Part ($b$) follows from \eqref{KLR-S4.4-P1-E3} and \eqref{KLR-S4.4-P1-E6}. The proof is complete.
\end{proof}

Due to \eqref{KLR-S4.4-P1-E3} we make the following observation: We examine $S(\beta_r)$ and $Z(r)$. For any $s=(s_1,s_2,\ldots)\in c_0$, $s\in S(\beta_r)$ is equivalent to $0\leq s_n\leq 2r,$ for $n=1,2,\ldots$. This implies that $S(\beta_r)\not\subseteq Z(r)$, $Z(r)\not\subseteq S(\beta_r) $ and $S(\beta_r)\cap Z(r)\ne \emptyset$. Therefore, from Proposition~\ref{KLR-S4.4-P1}, we obtain the following connections between the projections:
$$P_{c_0}(\beta_r)\nsubseteq \Pi_{c_0}(\beta_r),\quad \Pi_{c_0}(\beta_r)\nsubseteq P_{c_0}(\beta_r),\quad P_{c_0}(\beta_r)\cap \Pi_{c_0}(\beta_r)\ne \emptyset.$$
\section{Generalized Projections in $C[0, 1]$}\label{KLR-S5}
Let $C[0, 1]$ be the Banach space of all bounded continuous real-valued maps on $[0, 1]$ with norm
$$\|f\|=\max_{0\leq t\leq 1}|f(t)|,\ \ \text{for any}\ f\in C[0,1].$$
For any given positive integer $n$, let $\mathcal{P}_n$ be the subspace of $C[0, 1]$ that consists of all real coefficients polynomials of degree less than or equal to $n$.

For the metric projection $P_{\mathcal{P}_n}:C[0,1]\to \mathcal{P}_n$, in 1859, P. L. Chebyshev proved the following celebrated Chebyshev theorem that provides the criterion formulating the necessary and sufficient conditions for the metric projection from $C[0, 1]$ to $\mathcal{P}_n$.
\begin{thr}\label{KLR-S5-S1} If a function $f(t)$ is continuous on $[0,1]$ and if for $\displaystyle \sum_{k=0}^n a_kt^k\in \mathcal{P}_n$, we have
$$A=\max_{0\leq t\leq 1}\left|f(t)-\sum_{k=0}^n a_kt^k\right|,$$
then $\displaystyle \sum_{k=0}^na_kt^k$ is the polynomial of best uniform approximation for $f(t)$, that is,
$$\max_{0\leq t\leq 1}\left|f(t)-\sum_{k=0}^na_kt^k\right|=\min\left\{\max_{0\leq t\leq 1}\left|f(t)-\sum_{k=0}^nc_kt^k\right|:\ \sum_{k=0}^nc_kt^k\in \mathcal{P}_n\right\},$$
if and only if, there are $n+2$ points $0\leq t_0<t_1<\cdots t_{n+1}\leq 1$ such that
$$f(t_i)-\sum_{k=0}^n a_nt_i^k=\epsilon A(-1)^i,\ i=0,1,2,\ldots,n+1,\quad \epsilon\in \{-1,1\}.$$
\end{thr}
\subsection{The normalized duality map on $C[0,1]$}\label{KLR-S5.1}
We will now study the generalized metric projection $\Pi_{\mathcal{P}_n}:C[0,1]\to \mathcal{P}_n$. We recall that the dual space $C^*[0,1]$ of $C[0,1]$
 is $\text{rca}[0,1]$, in which the considered $\sigma$-field is the standard $\sigma$-field $\sum$ on $[0,1]$, including all closed and open sets of $[0,1]$. By the Riesz representation theorem, for any $\phi\in C^*[0,1]$, there is a real-valued and countable additive function $\mu\in rca[0,1]$ define on the given $\sigma$-field $\sum$ on $[0,1]$ such that
\begin{equation}\label{KLR-S5.1-E1}
\langle \phi,f\rangle =\int_0^1f(t)\mu (dt),\quad \text{for any}\ f\in C[0,1].
\end{equation}
The norm of $\phi$ in $C^*[0,1]$ is
$$\|\phi\|=\|\mu\|=\nu(\mu,[0,1]),$$
where $\nu(\mu,[0,1])$ is the total variation of $\mu$ on $[0,1]$. Throughout this subsection, without any special mention, we shall identify $\phi$ and $\mu$ via \eqref{KLR-S5.1-E1}.
We say that $\mu\in C^*[0,1]$ and \eqref{KLR-S5.1-E1} is rewritten as
\begin{equation}\label{KLR-S5.1-E2}
\langle \mu,f\rangle =\int_0^1f(t)\mu (dt),\quad \text{for any}\ f\in C[0,1].
\end{equation}
For any $f\in C[0,1]$, we define
$$M(f)=\{t\in [0,1]:\ |f(t)|=\|f\|\}.$$
The nonempty, closed subset $M(f)$ of $[0,1]$ is called the maximizing set of $f$.
\begin{lm}\label{KLR-S5.1-L1} For any $f\in C[0,1]$ with $\|f\|>0,$ the following implication holds:
$$\mu\in Jf\quad \Rightarrow\quad \nu(\mu,[0,1]\backslash M(f))=0.$$
\end{lm}
\begin{proof} Assume that $\mu\in Jf$ but $\nu(\mu,[0,1]\backslash M(f))>0.$ Then
$$\|f\|^2=\langle \mu,f\rangle =\|\mu\|^2=\nu(\mu,[0,1])^2,$$
implying that
\begin{align*}
\|f\|^2&=\langle \mu,f\rangle=\int_0^1f(t)\mu(dt)=\int_{M(f)}f(t)\mu(dt)+\int_{[0,1]\backslash M(f)} f(t)\mu(dt)\\
&\leq \|f\|\nu(\mu,M(f))+\int_{[0,1]\backslash M(f)} f(t)\mu(dt)\\
&<\|f\|\nu(\mu,M(f))+\|f\|\nu(\mu,[0,1]\backslash M(f))=\|f\|\nu(\mu,[0,1])=\|f\|^2,
\end{align*}
which is a contradiction. The proof is complete.
\end{proof}
\begin{lm}\label{KLR-S5.1-L2} Let $f\in C[0,1]$ be such that $\|f\|>0$. Assume either ($a$) or ($b$) below hold:
\begin{description}
\item[($a$)] For a positive integer $m$, let $\{t_1,t_2,\ldots,t_m\}\subset M(f)$ be distinct, and let $\alpha_1,\alpha_2,\ldots,\alpha_m$ be positive reals such that $\sum_{i=1}^m\alpha_i=1$. Let $\mu$ be a real function on $\sum$ such that for $i\in \{1,2,\ldots,m\}$,
    $$\mu(t_i)=\alpha_i\text{sign} f(t_i)\|f\|,\quad \text{and}\quad \nu(\mu,[0,1]\backslash \{t_1,t_2,\ldots,t_m\})=0.$$
\item[($b$)] For $[a,b]\subset M(f)$ with $0\leq a<b\leq 1$, we define $\mu$ on $\sum$ by
$$\mu[a,b]=\text{sign} f(a)\|f\|,\quad \text{and}\quad \nu(\mu,[0,1]\backslash [a,b])=0.$$
\end{description}
Then, $\mu\in rca[0,1]$ and $\mu\in Jf.$
\end{lm}
\begin{proof}
We will first prove the claim under ($a$). The proof of $\mu\in rca[0,1]$ is straightforward and hence omitted. We will now show that $\mu\in Jf.$ From $\{t_1,t_2,\ldots,t_m\}\subset M(f)$, we have
\begin{equation}\label{KLR-S5.1-L2P-E1}
|f(t_i)|=\|f\|,\quad \text{for}\ i=1,2,\ldots,m,
\end{equation}
which implies that
\begin{equation}\label{KLR-S5.1-L2P-E2}
\|\mu\|=\nu(\mu,[0,1])=\sum_{i=1}^m|\mu(t_i)|=\sum_{i=1}^m|\alpha_i\text{sign}f(t_i)\|f\||=\|f\|\sum_{i=1}^{m}\alpha_i=\|f\|.
\end{equation}
By \eqref{KLR-S5.1-L2P-E1} and \eqref{KLR-S5.1-L2P-E2}, we evaluate
\begin{equation}\label{KLR-S5.1-L2P-E3}
\langle \mu,f\rangle=\int_0^1f(t)\mu(dt)=\sum_{i=1}^mf(t_i)\alpha_i\text{sign}f(t_i)\|f\|=\sum_{i=1}^m |f(t_i)|\alpha_i\|f\|=\|f\|^2\sum_{i=1}^m\alpha_i=\|f\|^2.
\end{equation}
By combining \eqref{KLR-S5.1-L2P-E2} and \eqref{KLR-S5.1-L2P-E3}, we obtain $\mu\in Jf.$ The proof under ($b$) is similar and hence omitted. \end{proof}

By combining Lemma~\ref{KLR-S5.1-E1} and Lemma~\ref{KLR-S5.1-L1}, we have the following result:
\begin{pr}\label{KLR-S5.1-P1} For any $f\in C[0,1]$ with $\|f\|>0$, we have
\begin{description}
\item[($a$)] If $M(f)$ is not a singletom, then $Jf$ is an infinite subset of $C^*[0,1]$.
\item[($b$)] $Jf$ is a singleton \quad $\Leftrightarrow$ \quad $M(f)$ is a singleton
\end{description}
\end{pr}
\begin{proof} Part ($a$) follows from Lemma~\ref{KLR-S5.1-L1} immediately.

($b$). From ($a$), we have $Jf$ is a singleton implies that $M(f)$ is a singleton. On the other hand, assume that $M(f)$ is a singleton, say $M(f)=\{s\}$, for some $s\in [0,1]$. Then, for an arbitrary $\mu\in Jf$, by Lemma~\ref{KLR-S5.1-E1}, we obtain $\nu (\mu,[0,1]\backslash \{s\})=0$, which implies that
$$\|f\|^2=\langle \mu,f\rangle=\int_0^1f(t)\mu(dt)=\int_{\{s\}}f(t)\mu(dt)+ \int_{[0,1]\backslash \{s\}}f(t)\mu(dt)=\int_{\{s\}}f(t)\mu(dt)=f(s)\mu(\{s\}).$$
From $|f(s)|=\|f\|$, we get $\mu(\{s\})=\text{sign}(f(s))\|f\|,$ which completes the proof.
\end{proof}

For any $f,g\in C[0,1]$, we define
$$M(f\cap g)=\{t\in [0,1]:\ f(t)=g(t)\ \text{and}\ |f(t)|=\|f\|=|g(t)|=\|g\|\}.$$
It is clear that, for any $f,g\in C[0,1]$, we have the inclusion: $M(f\cap g)\subset M(f)\cap M(g).$
\begin{pr}\label{KLR-S5.1-P2} For any $f\in C[0,1]$ with $\|f\|>0$, we have
$$\mathcal{J}(f)\supseteq \{g\in C[0,1]:\ M(f\cap g)\ne \emptyset\}.$$
\end{pr}
\begin{proof} For any $g\in C[0,1]$, pick $s\in M(f\cap g)\ne \emptyset$. Then, similar to the proof of Proposition 4.3, define a real-valued and countable additive function $\mu$ on the given $\sigma$-field $\sum$ by
$$\mu(s)=\text{sign} f(s)\|f\|=\text{sign}g(s)\|g\|\quad\text{and}\quad \nu(\mu,[0,1]\backslash \{s\})=0.$$
Then $\mu\in rca[0,1]$ and $\mu \in Jf\cap Jg$, which ensures that $g\in \mathcal{J}(f).$
\end{proof}
\subsection{The generalized metric projection from  $C[0,1]$ to $\mathcal{P}_n$}\label{KLR-S5.2}
We will now study the generalized metric projection $\Pi_n:C\to 2^{\mathcal{P}_n}$. We will show that for any $f\in C[0,1]$ with $\|f\|>0$, $\Pi_{P_n}(f)$ is typically an infinite subset of $\mathcal{P}_n$ with cardinal number of $\aleph$, which is the cardinal number of the continuum. In this subsection, we denote by Card($A$), the cardinal number of a set $A$.
\begin{thr}\label{KLR-S5.2-T1} For any $f\in C[0,1]$ with $\|f\|>0$, the following statements hold:
\begin{description}
\item[($a$)] $\displaystyle \Pi_{\mathcal{P}_n}(f)\supseteq \mathcal{J}(f)\cap \mathcal{P}_n\supseteq \left\{\sum_{k=0}^n c_k t^k\in \mathcal{P}_n:\ M\left(f\cap \left(\sum_{k=0}^n c_k t^k\right)\right)\ne \emptyset\right\}.$
\item[($b$)] $\displaystyle \Pi_{\mathcal{P}_n}(f)\ne \emptyset$, that is, $\mathcal{P}_n$ is generalized proximal in $C[0,1]$ with respect to $\Pi_{\mathcal{P}_n}$
\item[($c$)] If $n=1$, then
\begin{align*}
M(f)\cap \{0,1\}=\emptyset\quad &\Rightarrow\quad 1\leq \text{Card}(\Pi_{\mathcal{P}_n}(f))\leq 2,\\
M(f)\cap \{0,1\}\ne\emptyset\quad &\Rightarrow\quad \aleph\leq \text{Card}(\Pi_{\mathcal{P}_n}(f)).
\end{align*}
\item[($d$)] If $n\geq 2$, then $\text{Card}(\Pi_{P_n}(f))\geq \aleph.$
\end{description}
\end{thr}
\begin{proof} The first part of (a) follows from (a) of Theorem 3.2 and the second part follows from Proposition~\ref{KLR-S5.1-P1}. We shall now prove (b). We take an arbitrary fixed point $s\in M(f)$ and  define a constant function $k$ by
$$k(t)=f(s),\ \ \text{for every}\ t\in [0,1].$$
We see that $k\in \mathcal{P}_n$ with degree zero, which satisfies $s\in M(f\cap k).$ Then by part (a), $k\in \Pi_{\mathcal{P}_n}(f).$

We will now prove (c) If $M(f)\cap \{0,1\}=\emptyset $, from (b) and taking $s\in M(f),$  the unique polynomial $\mathcal{P}_1$ passing through $(s,f(s))$ is a constant function $k(t)=f(s)$, for every $t\in [0,1]$. If besides $s\in M(f)$, $M(f)$ contains another point $u$ with $f(u)=-f(s)$, we define
$$\ell(t)=f(u),\quad \text{for every}\ t\in [0,1]$$
then $u\in M(f\cap \ell)$. It follows that $\Pi_{\mathcal{P}_n}(f)$ contains at least one element $k$ and at most two elements $k$ and $\ell$.

Now suppose that $M(f)\cap \{0,1\}\ne \emptyset$ Assume $a\in M(f)\cap \{0,1\}$. Let $b=1-a$. Then, for any given fixed $c$ with $-|f(a)|<c< |f(a)|$, define a linear function $k_c$ that passes through the points $(a,f(a))$ and $(b,c)$. Then $a\in M(f\cap k_c)$. By (a), it follows that
$$\Pi_{P_1}\supseteq \{k_c\in \mathcal{P}_1|\ c\in (-|f(a)|,|f(a)|)\}.$$

We now proceed to prove (d). If $M(f)\cap \{0,1\}\ne \emptyset,$ then (d) follows from (c). Therefore, we assume that $M(f)\cap \{0,1\}=\emptyset,$ Then, there is $v\in (0,1)$ such that $|f(v)|=\|f\|$, For arbitrary numbers $c$, and $d$ with $-\|f\|<c$, $d<\|f\|$, define a quadratic function $q_{cd}$ with vertex $(v,f(v))$ and passing through points $(0,c)$ and $(1,d)$. It satisfies that $v\in M(f\cap q_{cd})$. By (a), it follows that
$$\Pi_{\mathcal{P}_n}(f)\supseteq \{q_{cd}\in \mathcal{P}_n|\ -\|f\|<c,\ d<\|f\|\}$$
which completes  the proof.
\end{proof}

The following corollary is a new version of part (a) of Theorem~\ref{KLR-S5.2-T1}. We rewrite it here to strengthen our understanding of the properties of the generalized metric projection on $\mathcal{P}_n$ and to enhance the clarity of the strategies for finding the generalized metric projection points.
\begin{co}\label{KLR-S5.2-C1}Let $f\in C[0,1]$ with $\|f\|>0$. Let $\displaystyle \sum_{k=0}^nc_kt^k\in \mathcal{P}_n$ be a polynomial satisfying:
\begin{description}
\item[($a$)] $\displaystyle \left\|\sum_{k=0}^nc_nt^n\right\|=\|f\|$.
\item[($b$)]\ There is at least one point   $s\in [0,1]$ such that $\displaystyle \sum_{k=0}^nc_ks^k=f(s)$ and
$$|f(s)|=\|f\|=\left|\sum_{k=0}^nc_ks^k\right|=\left\|\sum_{k=0}^nc_kt^k\right\|.$$
\end{description}
Then: $\displaystyle \sum_{k=1}^nc_kt^k\in \Pi_{\mathcal{P}_n}(f).$
\end{co}
\begin{co}\label{KLR-S5.2-C2} For any $\displaystyle \sum_{k=0}^nc_kt^k\in \mathcal{P}_n$ with $\displaystyle \|\sum_{k=0}^nc_nt^n\|>0$, assume that
\begin{equation}\label{Co4.7E1}
M(f)\cap M\left(\sum_{k=0}^nc_kt^k\right)\ne \emptyset.
\end{equation}
Then:
$$\epsilon \frac{\|f\|}{\|\sum_{k=0}^nc_kt^k\|}\sum_{k=0}^nc_kt^n\in \Pi_{\mathcal{P}_n}(f),\quad \text{where}\ \epsilon\in \{-1,1\}.$$
\end{co}
\begin{proof} Since
$$\left\|\epsilon \frac{\|f\|}{\|\sum_{k=0}^nc_kt^k\|}\sum_{k=0}^nc_kt^k\right\|=\|f\|$$
and by \eqref{Co4.7E1}, we can select $\epsilon=1$ or $-1$ such that
$$M\left(f\cap \left(\epsilon \frac{\|f\|}{\|\sum_{k=0}^nc_kt^k\|}\sum_{k=0}^nc_kt^k\right)\right)\ne\emptyset$$
Then this corollary follows from (a) of Theorem 4.5.
\end{proof}
\section{Variational Characterizations of the metric projection}\label{KLR-S6}
We shall now provide some variational characterizations for the metric projection.
\begin{pr}\label{KLR-S6-P1} Let $B$ be a Banach space and let $C$ be a nonempty subset of $B$. For any given $x\in B$, if there is $j(x-z)\in J(x-z)$ such that
\begin{equation}\label{KLR-S6-P1-E1}
\langle j(x-z),z-y\rangle \geq 0,\quad \text{for all}\ y\in C.
\end{equation}
then $z\in P_{\text{\tiny{\emph{C}}}}(x)$.
\end{pr}
\begin{proof} By the property ($J_6$) of the normalized duality mapping on Banach spaces, for any $j(x-z)\in J(x-z)$ and for any $y\in C$, by \eqref{KLR-S6-P1-E1}, we have
$$\|x-y\|^2-\|x-z\|^2\geq 2\langle j(x-z),(x-y)-(x-z)\rangle=2\langle j(x-z),z-y\rangle\geq 0,\quad \text{for all}\ y\in C,$$
which confirms that $z\in P_{\text{\tiny{\emph{C}}}}(x).$
\end{proof}

\begin{pr}\label{KLR-S6-P2} Let $B$ be a Banach space and let $C$ be a nonempty, closed, and convex subset of $B$. For any given $x\in B\backslash C$, if $z\in P_{\text{\tiny{\emph{C}}}}(x)$, then there exists $j(x-z)\in J(x-z)$ such that
\begin{equation}\label{KLR-S6-P2-E1}
\langle j(x-z),z-y\rangle \geq 0,\quad \text{for all}\ y\in C.
\end{equation}
\end{pr}
\begin{proof} For $x\in B\backslash C$ and $z\in C$, if $z\in P_{\text{\tiny{\emph{C}}}}(x)$, by \cite[Theorem 3.8.4]{Zal02}, there exists $f\in B^*$ with $\|f\|=1$ such that
\begin{equation}\label{KLR-S6-P2P-E1}
\langle f,z-x\rangle=\|x-z\|\quad \text{and}\ \langle f,y-z\rangle\geq 0,\quad \text{for all}\ y\in C.
\end{equation}
We define $g\in B^*$ by
\begin{equation}\label{KLR-S6-P2P-E2}g=-\|x-z\|f,
\end{equation}
which due to $\|f\|=1$ implies that
\begin{equation}\label{KLR-S6-P2P-E3}\|g\|=\|-\|x-z\|\,f\|=\|x-z\|\|f\|=\|x-z\|.
\end{equation}
Then, \eqref{KLR-S6-P2P-E1} and \eqref{KLR-S6-P2P-E2} yield
\begin{equation}\label{KLR-S6-P2P-E4}\langle g,x-z\rangle=\langle -\|x-z\|f,x-z\rangle=\|x-z\|\langle f,z-x\rangle=\|x-z\|^2
\end{equation}
By combining \eqref{KLR-S6-P2P-E3} and \eqref{KLR-S6-P2P-E4}, we get $g\in J(x-z).$

Moreover, by the second part of \eqref{KLR-S6-P2P-E1} and \eqref{KLR-S6-P2P-E2},  we have
$$\langle g,z-y\rangle=\langle -\|x-z\|f,z-y\rangle=\|x-z\|\langle f,y-z\rangle\geq 0,\quad \text{for all}\ y\in C. $$
Then  \eqref{KLR-S6-P1-E1}  is proved by simply rewriting $g=j(x-z)\in J(x-z).$
\end{proof}

By combining Proposition~\ref{KLR-S6-P1} and Proposition~\ref{KLR-S6-P2}, we obtain
\begin{thr}\label{KLR-S6-T1} Let $B$ be a Banach space and let $C$ be a nonempty, closed, and convex subset of $B$. For any given $x\in B\backslash C$ and $z\in C$, we have $z\in P_{\text{\tiny{\emph{C}}}}(x)$ if and only if there is $j(x-z)\in J(x-z)$ such that
$$\langle j(x-z),z-y\rangle\geq 0,\quad \text{for all}\ y\in C. $$
\end{thr}
\section{Concluding Remarks}\label{KLR-S7}
We carried out an extensive study of the generalized projection and the generalized metric projection in general Banach spaces for arbitrary sets.
We provided illustrative examples to shed some light on the structure of the proposed notions. It would be of interest to develop iterative solvers for variational inequalities and convex feasibility problem based on the projections in general Banach spaces (see \cite{AlbBut97}). Moreover, the stochastic approximation approach has been extensively studied for stochastic variational inequalities and stochastic inverse problems in recent years (see \cite{AlbChiLi12,BarRoyStr07,Rachel2020}). However, there are no explorations of these topics in Banach spaces, and it seems natural to use generalized projections to address these issues. Projection methods have also been used for vector optimization problems but only in the framework of Hilbert spaces (see \cite{ZhaKobYao21});  their extensions to Banach spaces can naturally benefit from the proposed notions.\\[5pt]
\textbf{Acknowledgment.} Simeon Reich was partially supported by the Israel Science Foundation (Grant 820/17), the Fund for the Promotion of Research at the Technion and by the Technion General Research Fund.
\section*{Appendix}\label{KLR-S8}
We list some properties below for easy reference, for details see \cite{Cio90,Rei92,Tak00}.
\begin{lm}\label{KLR-S8-L1}
Let $B$ be a Banach space, let $B^*$ be the dual of $B^*$, let $\langle \cdot,\cdot\rangle$ be the duality paoring, and let $J$ is the normalized duality map. Then the following properties hold:
\begin{description}
\item[($J_1$)] For any $x\in B$, $J(x)$ is nonempty, bounded, closed, and convex subset of $B^*$.
\item[($J_2$)] $J$ is the identity operator in any Hilbert space $H$, that is, $J = I_H$.
\item[($J_3$)] $J(\theta)=\theta$, where $\theta$ denotes the origin.
\item[($J_4$)] For any $x\in B$ and a real number $\alpha$, we have $J(\alpha x)=\alpha J(x).$
\item[($J_5$)] For any $x,y\in B$, $\phi\in J(x)$ and $\psi\in J(y)$, we have $\langle \phi-\psi,x-y\rangle\geq 0. $
\item[($J_6$)] For any $x,y\in B$ and $\psi\in J(y),$ we have $2\langle \psi,x-y\rangle\leq \|x\|^2-\|y\|^2$.
\item[($J_7$)] $B$ is strictly convex, if and only if, $J$ is one-to-one, that is, for  any $x,y\in B$, $x\ne y$ implies $J(x)\cap J(y)=\emptyset.$
\item[($J_8$)] $B$ is strictly convex, if and only if, $J$ is strictly monotone, that is, for  any $x,y\in B$ with $x\ne y$, and for $\psi\in J(x)$ and $\psi\in J(y)$, we have $\langle \phi-psi,x-y\rangle>0. $
\item[($J_9$)] $B$ is strictly convex, if and only if, for  any $x,y\in B$ with $\|x\|=\|y\|=1$ and $x\ne y$, $\phi\in J(x)$ implies that $\langle \phi,y\rangle <1$.
\item[($J_{10}$)] If $B^*$ is strictly convex, then $J$ is a single-valued map.
\item[($J_{11}$)] $B$ is smooth, if and only if, $J$ is a single-valued map.
\item[($J_{12}$)] If $J$ is a single-valued mapping, then $J$ is a norm to weak$^*$ continuous.
\item[($J_{13}$)]  $B$ is reflexive if and only if,  $J$ is a map of $B$ onto $B^*$.
\item[($J_{14}$)]  If $B$ is smooth, then $J$ is continuous.
\item[($J_{15}$)]  $J$ is a uniformly continuous  on each bounded set in uniformly smooth Banach spaces.
\end{description}
\end{lm}
}
\printbibliography
\end{document}